\documentclass[onecolumn]{IEEEtran}
\usepackage{verbatim} 
\usepackage{subfiles, etoolbox}
\usepackage{amsthm}
\usepackage{amsfonts}
\usepackage{mathtools}
\usepackage{import}

\usepackage{color}
\usepackage[normalem]{ulem}
\usepackage{amsmath}

\definecolor{orange}{RGB}{255,127,0}

\usepackage[ruled,vlined]{algorithm2e}
\usepackage{subfig}
\usepackage[colorlinks=true]{hyperref} 
\usepackage{cite}
\usepackage{proof-at-the-end}

\def\*#1{\mathbf{#1}}
\newtheorem{thm}{Theorem}[section]
\newtheorem{assum}[thm]{Assumption}
\newtheorem{lem}[thm]{Lemma}
\newtheorem{cor}[thm]{Corollary}

\newtheorem{example}{Example}

\SetKwInput{KwInit}{Initialization} 
\SetKwInput{KwInput}{Input} 
\DeclarePairedDelimiter\ceil{\lceil}{\rceil}
\DeclarePairedDelimiter\floor{\lfloor}{\rfloor}

\begin{document}

\title{S-NEAR-DGD: A Flexible Distributed Stochastic Gradient Method for Inexact Communication}

\author{Charikleia~Iakovidou, Ermin~Wei
\thanks{{C. Iakovidou  and E. Wei are with the Department of Electrical and Computer Engineering, Northwestern University,
        Evanston, IL USA,  
        {\tt\small chariako@u.northwestern.edu, ermin.wei@northwestern.edu}}}
}


\maketitle

\begin{abstract}


We present and analyze a stochastic  distributed method (S-NEAR-DGD) that can tolerate inexact computation and inaccurate information exchange to alleviate the problems of costly gradient evaluations and bandwidth-limited communication in large-scale systems. Our method is based on a class of flexible, distributed first order algorithms that allow for the trade-off of computation and communication to best accommodate the application setting. We assume that all the information exchange between nodes is subject to random distortion and that only stochastic approximations of the true gradients are available. Our theoretical results prove that the proposed algorithm converges linearly in expectation to a neighborhood of the optimal solution for strongly convex objective functions with Lipschitz gradients. We characterize the dependence of this neighborhood on algorithm and network parameters, the quality of the communication channel and the precision of the stochastic gradient approximations used. Finally, we provide numerical results to evaluate the empirical performance of our method.

\end{abstract}

\IEEEpeerreviewmaketitle

\section{Introduction}

The study of distributed optimization algorithms has been an area of intensive research for more than three decades. The need to harness the computing power of multiprocessors to boost performance and solve increasingly complex problems~\cite{tsitsiklis_Problems_1984,bertsekas_parallel}, and the emergence of a multitude of networked systems that lack central coordination such as wireless sensors networks~\cite{predd_ml_2009,sensor4,sensor5}, power systems~\cite{power5,smart_grids,giannakisPower} and multi-robot and multi-vehicle networks~\cite{vehicles,robots2,robots3}, necessitated the development of optimization algorithms that can be implemented in a distributed manner. Moreover, the proliferation of data in recent years coupled with storage constraints, growing computation costs and privacy concerns have sparked significant interest in decentralized optimization for machine learning~\cite{nedic_ML_review,hong2016unified,richtarik2016parallel}.

One of the most well-studied settings in distributed optimization is often referred to as {\it consensus optimization} ~\cite{bertsekas_parallel}. Consider a connected, undirected network of~$n$ nodes $\mathcal{G}(\mathcal{V},\mathcal{E})$, where $\mathcal{V}$ and $\mathcal{E}$ denote the sets of nodes and edges, respectively. The goal is to collectively solve a decomposable optimization problem,
\begin{equation}
    \min_{x \in \mathbb{R}^p} f(x) = \sum_{i=1}^n f_i(x),
    \label{eq:prob_orig}
\end{equation}
where $x \in \mathbb{R}^p$ and $f_i: \mathbb{R}^p \rightarrow \mathbb{R}$.
Each function $f_i$ is private to node $i$. To enable distributed computation, each node maintains a local copy $x_i$ in $\mathbb{R}^p$ to approximate the global variable $x$. Problem~(\ref{eq:prob_orig}) can then be equivalently reformulated into the following,
\begin{equation}
    \label{eq:consensus_prob}
    \begin{split}
    \min_{x_i \in \mathbb{R}^{p}} \sum_{i=1}^n f_i(x_i), \quad \text{s.t. } x_i=x_j, \quad \forall(i,j) \in \mathcal{E}.
    \end{split}
\end{equation}
A similar problem setup was studied as far back as in~\cite{tsitsiklis_Problems_1984,tsitsiklis_distributed,bertsekas_parallel}. A well-known iterative method for this problem, the Distributed (Sub)Gradient Descent (DGD) method~\cite{NedicSubgradientConsensus}, involves taking local gradient steps and weighted averages with neighbors at each iteration. A class of algorithms known as gradient tracking methods, which include EXTRA and DIGing, can be viewed as an improved version of DGD with an additional step of averaging gradient information amongst neighbors, and can achieve exact convergence under constant stepsize~\cite{diging,extra,exact_diffusion,harnessing_smoothness}. 
A recent summary of several distributed methods for solving problem~(\ref{eq:consensus_prob})  can be found in~\cite{survey_distr_opt}.


Distributed optimization algorithms typically rely on some combination of local computation steps, where the nodes aim to decrease their local functions, and consensus (or communication) steps, where nodes exchange information with their neighbors over the network. The amounts of computation and communication executed at every iteration are usually fixed for a given algorithm. However due to the diversity of distributed optimization applications, a "one size fits all" approach is unlikely to achieve optimal performance in every setting. The goal of this work is to study and develop a flexible, fast and efficient distributed method that can be customized depending on application-specific requirements and limitations. 

In addition to flexibility, our proposed method can also address two major challenges for distributed optimization methods: communication bottlenecks and costly gradient evaluations. Next, we summarize some the existing techniques to tackle these two issues. 


\subsection{Literature review}

\subsubsection{Distributed optimization algorithms with quantized communication}

The amount of communication between nodes has long been identified as a major performance bottleneck in decentralized computing, especially as the volume and dimensionality of available data increase~\cite{reisizadeh_exact_2018,lan_communication-efficient_2017}. 
Moreover, in any practical setting where the bandwidth of the communication channel is limited, the information exchanged cannot be represented by real-valued vectors with arbitrary precision.  
Both of these concerns motivate us to design distributed optimization methods where nodes receive inexact/quantized information from their neighbors.
Distributed methods that can work with quantized communication include incremental algorithms~\cite{rabbat_quantized_2005}, DGD~\cite{nedic_distributed_2008,li_distributed_2017} and the dual averaging method~\cite{yuan_distributed_2012}.
Different approaches have been proposed to guide distributed algorithms with inexact communication towards optimality, such as using weighted averages of incoming quantized information and local estimates~\cite{reisizadeh_exact_2018,doan_distributed_2018}, designing custom quantizers~\cite{lee2018finite,doan_accelerating_2018,pu_quantization_2017}, employing encoding/decoding schemes~\cite{alistarh2017qsgd,reisizadeh_exact_2018,yi_quantized_2014}, and utilizing error correction terms~\cite{zhu_param_estim,lee2018finite,doan_accelerating_2018}. Among these, only~\cite{pu_quantization_2017,lee2018finite} achieve exact geometric convergence by employing dynamic quantizers that require the tuning of additional parameters and global information at initialization. However, neither of these methods allow for adjusting the amounts of computation and communication executed at every iteration.



\subsubsection{Stochastic gradient in distributed optimization}

The computation of exact gradients for some large-scale problems can be prohibitively expensive and hence stochastic approximations obtained by sampling a subset of data are often used instead. Various studies and analyses on stochastic gradient based methods have been done in centralized settings ~\cite{sgd_bottou, goyal_minibatch}, federated learning (client-server model)~\cite{federated5,parallel_sgd} and distributed settings over general topologies~\cite{lian_can_nodate, morral_success_2017,pu_central_distr}, which is the setting our work adopts. 
In this particular setting, existing approaches include stochastic variants of DGD~\cite{sundharram_distributed_2010,srivastava_distributed_2011,lian_can_nodate}, 
stochastic diffusion algorithms~\cite{towfic_adaptive_2014,morral_success_2017,olshevsky_non-asymptotic_2019}, primal-dual methods~\cite{chatzipanagiotis_distributed_2016,lan_communication-efficient_2017,hong_stochastic_2017}, gradient-push algorithms~\cite{nedic_stochastic_2014,olshevsky_robust_2018}, dual-averaging methods~\cite{duchi_ml_2012}, accelerated distributed algorithms~\cite{fallah2019robust} 
and stochastic distributed gradient tracking methods ~\cite{mokhtari_dsa:_2015, pu_distributed_2018-1,shen_towards_2018,variance_reduced_gt,li2020sdiging}. While some of these methods can achieve exact convergence (in expectation) with linear rates (eg. stochastic variants of gradient tracking, exact diffusion), they need to be combined with variance reduction techniques that may have excessive memory requirements. Moreover, all of the aforementioned algorithms have a fixed structure and lack adaptability to different cost environments.

\subsection{Contributions}

In this paper, we propose and analyze a distributed first order algorithm, the Stochastic-NEAR-DGD (S-NEAR-DGD) method, that uses stochastic gradient approximations and tolerates the exchange of noisy information to save on bandwidth and computational resources. Our method is based on a class of flexible algorithms (NEAR-DGD)~\cite{berahas_balancing_2019} that permit the trade-off of computation and communication to best accommodate the application setting. In this work, we generalize our previous results analyzing NEAR-DGD in the presence of either deterministically quantized communication~\cite{berahas2019nested} or stochastic gradient errors~\cite{Iakovidou2019NestedDG}, and unify them under a common, fully stochastic framework. We provide theoretical results to demonstrate that S-NEAR-DGD converges to a neighborhood of the optimal solution with geometric rate, and that if an error-correction mechanism is incorporated to consensus, then the total communication error induced by inexact communication is independent of the number of consensus rounds peformed by our algorithm. Finally, we empirically show by conducting a series of numerical experiments that S-NEAR-DGD performs comparably or better than state-of-the-art methods, depending on how quantized consensus is implemented.

The rest of the paper is organized as follows. In Section \ref{sec:algo}, we introduce the S-NEAR-DGD method. Next, we analyze the convergence properties of S-NEAR-DGD in Section~\ref{sec:analysis}. We present our numerical results in Section~\ref{sec:numerical} and conclude this work in Section~\ref{sec:conclusion}.

\subsection{Notation}
In this paper, all vectors are column vectors. The concatenation of local vectors $v_{i}$ in $\mathbb{R}^p$ is denoted by $\mathbf{v} = [v_i]_{i=\{1,2,\ldots,n\}}$ in $\mathbb{R}^{np}$ with a lowercase boldface letter. We use uppercase boldface letters for matrices. We will use the notations $I_p$ and $1_n$ for  the identity matrix of dimension $p$ and the vector of ones of dimension $n$, respectively. The element in the $i$-th row and $j$-th column of a matrix $\*H$ will be denoted by $h_{ij}$, and the $p$-th element of a vector $v$ by $\left[v\right]_p$. The transpose of a vector $v$ will be denoted by $v^T$. We will use $\|\cdot\|$ to denote the $l_2$-norm, i.e. for $v \in \mathbb{R}^p$ $\left\|v\right\|= \sqrt{\sum_{i=1}^p \left[v\right]_i^2}$, and $\langle v, u \rangle$ to denote the inner product of two vectors $v,u$. We will use $\otimes$ to denote the Kronecker product operation. Finally, we define $\mathcal{N}_i$  to be the set of neighbors of node $i$, i.e., $\mathcal{N}_i = \{j \in \mathcal{V}: (i,j) \in \mathcal{E}\}$.

\section{The S-NEAR-DGD method}
\label{sec:algo} 

In this section, we first introduce a few standard technical assumptions and notation on our problem setting, followed by a quick review of the NEAR-DGD method, which serves as the main building block of the proposed method. Finally, we present the S-NEAR-DGD method. We adopt the following standard assumptions on the local functions $f_i$ 
of problem~\eqref{eq:consensus_prob}.
\begin{assum} \textbf{(Local Lipschitz gradients)}
\label{assum:lip}
Each local objective function $f_i$ has \mbox{$L_i$-Lipschitz} continuous gradients, i.e. $ \|\nabla f_i(x) - \nabla f_i(y)\| \leq L_i\|x-y\|,\quad \forall x,y \in \mathbb{R}^p$.
\end{assum}
\begin{assum} \textbf{(Local strong convexity)}
\label{assum:conv}
Each local objective function $f_i$ is \mbox{$\mu_i$-strongly} convex, i.e. $ f_i(y) \geq f_i(x) + \langle \nabla f_i(x), y-x \rangle + \frac{\mu_i}{2}\|x-y\|_2^2,\quad \forall x,y \in \mathbb{R}^p$.
\end{assum}
We can equivalently rewrite problem~\eqref{eq:consensus_prob} in the following compact way,
\begin{equation}
    \label{eq:consensus_prob2}
    \begin{split}
    \min_{\*x \in \mathbb{R}^{np}} \*f\left(\*x\right)=\sum_{i=1}^n f_i(x_i), &\quad \text{s.t. } \left(\*W \otimes I_p \right)\*x=\*x,
    \end{split}
\end{equation}
where $\*x=[x_i]_{i=\{1,2,\ldots,n\}}$ in $\mathbb{R}^{np}$ is the concatenation of local variables $x_i$ , $\*f:\mathbb{R}^{np}\rightarrow \mathbb{R}$ and $\*W \in \mathbb{R}^{n \times n}$ is a matrix satisfying the following condition.
\begin{assum} \textbf{(Consensus matrix)}
\label{assum:consensus_mat} The matrix $\*W \in \mathbb{R}^{n \times n}$
has the following properties: i) symmetry, ii) double stochasticity, and iii) $w_{i,j} > 0$ if and only if $j \in \mathcal{N}_i$ or $i=j$ and $w_{i,j} = 0$ otherwise.
\end{assum}
We will refer to $\*W$ as the \emph{consensus matrix} throughout this work. Since $\*W$ is symmetric it has $n$ real eigenvalues, which we order by $\lambda_n \leq \lambda_{n-1} \leq ... \leq \lambda_2 \leq \lambda_1$ in ascending order. Assumption~\ref{assum:consensus_mat} implies that $\lambda_1 = 1$ and $\lambda_2<\lambda_1$ for any connected network. The remaining eigenvalues have absolute values strictly less than 1, i.e., $-1<\lambda_n$. Moreover, the equality $\left(\*W \otimes I_p \right)\*x=\*x$ holds if and only if $x_i=x_j$ for all $j \in \mathcal{N}_i$~\cite{NedicSubgradientConsensus}, which establishes  the equivalence between the two formulations~(\ref{eq:consensus_prob}) and~(\ref{eq:consensus_prob2}). 

We will refer to the absolute value of the eigenvalue with the second largest absolute value of $\*W$ as $\beta$, i.e. $\beta = \max\left\{|\lambda_2|,|\lambda_n|\right\}$. The iteration $\*x^{k+1}=\left(\*W \otimes I_p \right)\*x^{k},$ commonly referred to as the {\it consensus} step, can be implemented in a distributed way where each node exchanges information with neighbors and updates its private value of $x_i$ by taking a weighted average of its local neighborhood. When taking $k$ to infinity, all nodes in the network converge to the same value and thus reach consensus. It has been shown that $\beta$ controls the speed of consensus, with smaller values yielding faster convergence~\cite{boyd_consensus}.


\subsection{The NEAR-DGD method}
We next review the NEAR-DGD method~\cite{berahas_balancing_2019}, on which this work is based. NEAR-DGD is an iterative distributed method, where at every iteration, each node first 
decreases its private cost function by taking a local gradient (or \emph{computation}) step. Next, it performs a number of nested consensus (or \emph{communication}) steps. We denote the local variable at node $i$, iteration count $k$ and consensus round $j$ by $v_{i,k}^j \in \mathbb{R}^p$ (if the variable $v$ is not updated by consensus, the superscript will be omitted). We will use the notation $\bar{v}_k:=\sum_{i=1}^n v_{i,k}$ to refer to the average of $v_{i,k}$ across nodes. 

Starting from arbitrary initial points $x^{t(0)}_{i,0}=y_{i,0} \in \mathbb{R}^{p}$, the local iterates of NEAR-DGD at iteration $k=1,2,...$ can be expressed as
\begin{subequations}
\begin{align}
&y_{i,k}= x^{t(k-1)}_{i,k-1} - \alpha \nabla f_i \left(x^{t(k-1)}_{i,k-1}\right),\label{eq:near_dgd_local_y_ORIG}\\
&x^{t(k)}_{i,k} = \sum_{l=1}^n w^{t(k)}_{il} y_{l,k},\label{eq:near_dgd_local_x_ORIG}
\end{align}
\end{subequations}
where $t(k)$ is the number of consensus rounds performed at iteration $k$, $\alpha>0$ is a positive steplength, and 
\begin{equation*}
    \*W^{t(k)}=\underbrace{\*W \cdot \*W \cdot ... \cdot \*W}_{t(k)\text{ times}} \in \mathbb{R}^{n \times n}.
\end{equation*}
The flexibility of NEAR-DGD lies in the sequence of consensus rounds per iteration $\{t(k)\}$, which can be tuned depending on the deployed environment to balance convergence accuracy/speed and total application cost. With an increasing sequence of $\{t(k)\}$, NEAR-DGD can achieve exact convergence to the optimal solution of problem~\eqref{eq:prob_orig} and with $t(k) = k$ it converges linearly (in terms of gradient evaluations) for strongly convex functions~\cite{berahas_balancing_2019}.

\subsection{The S-NEAR-DGD method}
To accommodate bandwidth-limited communication channel,  we assume that whenever node $i \in \{1,...,n\}$ needs to communicate a vector $v \in \mathbb{R}^p$ to its neighbors, it sends an approximate vector $\mathcal{T}_c\left[v\right]$ instead, i.e., $\mathcal{T}_c\left[\cdot\right]$ is a randomized operator which modifies the input vector to reduce the bandwidth. Similarly, to model the availability of only inexact gradient information, we assume that instead of the true local gradient $\nabla f_i\left(x_i\right)$, node $i$ calculates an approximation $\mathcal{T}_g\left[\nabla f_i \left(x_i\right)\right]$, where $\mathcal{T}_g\left[\cdot\right]$ is a randomized operator denoting the inexact computation. We refer to this method with inexact communication and gradient computation as the Stochastic-NEAR-DGD (S-NEAR-DGD) method. 
\begin{algorithm}
\SetAlgoLined
\KwInit{ Pick $x^{t(0)}_{i,0}=y_{i,0}$}
\For{$k=1,2,...$}{
Compute $g_{i,k-1}=\mathcal{T}_g\left[\nabla f_i \left(x^{t(k-1)}_{i,k-1}\right)\right]$\\
Update $y_{i,k}\leftarrow x^{t(k-1)}_{i,k-1} - \alpha g_{i,k-1}$\\
Set $x^0_{i,k}= y_{i,k}$\\
\For{$j=1,...,t(k)$}{
Send $q^{j}_{i,k}=\mathcal{T}_c\left[x^{j-1}_{i,k}\right]$ to neighbors $l \in \mathcal{N}_i$ and receive $q^{j}_{l,k}$ \\ 
Update $x^j_{i,k} \leftarrow \sum_{l=1}^n \left(w_{il}  q^{j}_{l,k} \right) + \left( x^{j-1}_{i,k} - q^j_{i,k}\right)$
}
}
\caption{S-NEAR-DGD at node $i$}
\label{algo:sneardgd}
\end{algorithm}

Each node $i \in \{1,...,n\}$ initializes and preserves the local variables $x^{j}_{i,k}$ and $y_{i,k}$. At iteration $k$ of S-NEAR-DGD, node $i$ calculates the stochastic gradient approximation $g_{i,k-1}=\mathcal{T}_g\left[\nabla f_i \left(x^{t(k-1)}_{i,k-1}\right)\right]$ and uses it to take a local gradient step and update its internal variable $y_{i,k}$. Next, it sets $x^0_{i,k}= y_{i,k}$ and performs $t(k)$ nested consensus rounds, where during each consensus round $j \in \{1,...,t(k)\}$ it constructs the bandwidth-efficient vector $q^{j}_{i,k}=\mathcal{T}_c\left[x^{j-1}_{i,k}\right]$, forwards it to its neighboring nodes $l \in \mathcal{N}_i$ and receives the vectors $q^j_{l,k}$ from neighbors. Finally, during each consensus round, node $i$ updates its local variable $x^j_{i,k}$ by forming a weighted average of the vectors $q^j_{l,k}$, $l=1,...,n$ and adding the residual error correction term $\left(x^{j-1}_{i,k} - q^j_{i,k}\right)$. The entire procedure is presented in Algorithm~\ref{algo:sneardgd}.

Let $\*x^{t(0)}_0=\*y_0=\left[y_{1,0};...;y_{n,0}\right]$ be the concatenation of local initial points $y_{i,0}$ at nodes $i=1,...,n$ as defined in Algorithm~\ref{algo:sneardgd}. The system-wide iterates of S-NEAR-DGD at iteration count $k$ and $j$-th consensus round can be written compactly as,
\begin{subequations}
\begin{align}
    &\*y_{k} = \*x_{k-1}^{t(k-1)} - \alpha \*g_{k-1}\label{eq:near_dgd_y},\\
    &\*x^j_k =  \*x^{j-1}_k +\left(\*Z - I_{np}\right)\*q^{j}_k,\quad j=1,...,t(k), \label{eq:near_dgd_x}
\end{align}
\end{subequations}
where $\*x^0_k=\*y_k$, $\*Z = \left(\*W \otimes I_p\right) \in \mathbb{R}^{np \times np}$, $g_{i,k-1}=\mathcal{T}_g\left[\nabla f_i \left(x^{t(k-1)}_{i,k-1}\right)\right]$, $q^{j}_{i,k}=\mathcal{T}_c\left[x^{j-1}_{i,k}\right]$ for $j=1,...,t(k)$ and $\*g_{k-1}$ and $\*q^j_k$ are the long vectors formed by concatenating $g_{i,k-1}$ and $q^{j}_{i,k}$ over $i$ respectively. 

Moreover, due to the double stochasticity of $\*W$, the following relations hold for the average iterates $\bar{y}_k=\frac{1}{n}\sum_{i=1}^n y_{i,k}$ and $\bar{x}^j_k=\frac{1}{n}\sum_{i=1}^n x^j_{i,k}$ for all $k$ and $j$,
\begin{subequations}
\begin{align}
    &\bar{y}_{k} = \bar{x}_{k-1}^{t(k-1)} - \alpha \bar{g}_{k-1}\label{eq:near_dgd_y_avg},\\
    &\bar{x}^j_k =  \bar{x}^{j-1}_k,\quad j=1,...,t(k), \label{eq:near_dgd_x_avg}
\end{align}
\end{subequations}
where $\bar{g}_{k-1}=\frac{1}{n}\sum_{i=1}^n g_{i,k-1}$.

The operators $\mathcal{T}_c\left[\cdot\right]$ and $\mathcal{T}_g\left[\cdot\right]$ can be interpreted as $ \mathcal{T}_c \left[x^{j-1}_{i,k}\right]=x^{j-1}_{i,k}+\epsilon^{j}_{i,k}$, and $\mathcal{T}_g\left[\nabla f_i \left(x^{t(k-1)}_{i,k-1}\right)\right]=\nabla f_i \left(x^{t(k-1)}_{i,k-1}\right)+\zeta_{i,k}$, where $\epsilon^{j}_{i,k}$ and $\zeta_{i,k}$ are random error vectors. We list our assumptions on these vectors and the operators $\mathcal{T}_c\left[\cdot\right]$ and $\mathcal{T}_g\left[\cdot\right]$ below. 

\begin{assum}\textbf{(Properties of $\mathcal{T}_c\left[\cdot\right]$)} 
\label{assum:tc_bound} The operator $\mathcal{T}_c\left[\cdot\right]$ is iid for all $i=1,...,n$, $j=1,...,t(k)$ and $k\geq1$. Moreover, the errors $\epsilon^{j}_{i,k}=\mathcal{T}_c \left[x^{j-1}_{i,k}\right]-x^{j-1}_{i,k}$ have zero mean and bounded variance for all $i=1,...,n$, $j=1,...,t(k)$ and $k\geq1$, i.e.,
\begin{gather*}
  \mathbb{E}_{\mathcal{T}_c}\left[\epsilon^j_{i,k}\big|x^{j-1}_{i,k}\right]=0, \quad \mathbb{E}_{\mathcal{T}_c}\left[\|\epsilon^j_{i,k}\|^2 \big |x^{j-1}_{i,k}\right] \leq \sigma_c^2,
\end{gather*}
where $\sigma_c$ is a positive constant and the expectation is taken over the randomness of $\mathcal{T}_c$. 
\end{assum}

\begin{example}\textbf{(Probabilistic quantizer)}
\label{ex:prob_q}

An example of an operator satisfying Assumption~\ref{assum:tc_bound} is the probabilistic quantizer in~\cite{yuan_distributed_2012}, defined as follows: for a scalar $x \in \mathbb{R}$, its quantized value $\mathcal{Q}\left[x\right]$ is given by
\begin{equation*}
    \mathcal{Q}\left[x\right]=\begin{cases} \floor*{x} \quad \text{with probability $\left(\ceil*{x}-x\right)\Delta$} \\ \ceil*{x} \quad \text{with probability $\left(x-\floor*{x}\right)\Delta$},\end{cases}
\end{equation*}
where $\floor*{x}$ and  $\ceil*{x}$ denote the operations of rounding down and up to the nearest integer multiple of $1/\Delta$, respectively, and $\Delta$ is a positive integer.
\end{example}
It is shown in~\cite{yuan_distributed_2012} that $\mathbb{E}\Big[x-\mathcal{Q}\left[x\right]\Big] =0$ and $\mathbb{E}\left[\left|x-\mathcal{Q}\left[x\right]\right|^2\right] \leq \frac{1}{4\Delta^2}$. For any vector $v = [v_i]_{i=\{1,\ldots,p\}}$ in $\mathbb{R}^p$, we can then apply the operator $\mathcal{Q}$ element-wise to obtain $\mathcal{T}_c\left[v\right]=\Big[\mathcal{Q}\left[v_i\right]\Big]_{i=\{1,\ldots,p\}}$ in $\mathbb{R}^p$ with $\mathbb{E}_{\mathcal{T}_c}\left[v - \mathcal{T}_c\left[v\right]\big| v\right] =\mathbf{0}$ and $\mathbb{E}_{\mathcal{T}_c}\left[\left\|v - \mathcal{T}_c\left[v\right]\right\|^2 \big| v\right] \leq \frac{p}{4\Delta^2}=\sigma^2_c$.

\begin{assum}\textbf{(Properties of $\mathcal{T}_g\left[\cdot\right]$)}
\label{assum:tg_bound}  The operator $\mathcal{T}_g\left[\cdot\right]$ is iid for all $i=1,...,n$ and $k\geq1$. Moreover, the errors $\zeta_{i,k}=\mathcal{T}_g\left[\nabla f_i \left(x^{t(k-1)}_{i,k-1}\right)\right]-\nabla f_i \left(x^{t(k-1)}_{i,k-1}\right)$ have zero mean and bounded variance for all $i=1,...,n$ and $k\geq1$,
\begin{gather*}
  \mathbb{E}_{\mathcal{T}_g}\left[\zeta_{i,k} \Big|x^{t(k-1)}_{i,k-1}\right]=0, \quad \mathbb{E}_{\mathcal{T}_g}\left[\left\|\zeta_{i,k}\right\|^2 \Big |x^{t(k-1)}_{i,k-1}\right] \leq \sigma_g^2,
\end{gather*}
where $\sigma_g$ is a positive constant and the expectation is taken over the randomness of $\mathcal{T}_g$.
\end{assum}
Assumption~\ref{assum:tg_bound} is standard in the analysis of distributed stochastic gradient methods~\cite{sundharram_distributed_2010,nedic_stochastic_2014,pu_distributed_2018-1,lian_can_nodate}.


We make one final assumption on the independence of the operators $\mathcal{T}_c\left[\cdot\right]$ and $\mathcal{T}_g\left[\cdot\right]$, namely that the process of generating stochastic gradient approximations does not affect the process of random quantization and vice versa.

\begin{assum}\textbf{(Independence)}
\label{assum:iid}
The operators $\mathcal{T}_g\left[\cdot\right]$ and $\mathcal{T}_c\left[\cdot\right]$ are  independent for all $i=1,...,n$, $j=1,...,t(k)$ and $k\geq1$.


\end{assum}


Before we conclude this section, we note that there are many possible choices for the operators $\mathcal{T}_c\left[\cdot\right]$ and $\mathcal{T}_g\left[\cdot\right]$ and each would yield a different algorithm instance in the family of NEAR-DGD-based methods. For example, both $\mathcal{T}_c\left[\cdot\right]$ and $\mathcal{T}_g\left[\cdot\right]$ can be identity operators as in~\cite{berahas_balancing_2019}. We considered quantized communication using deterministic (D) algorithms (e.g. rounding to the nearest integer with no uncertainty) in~\cite{berahas2019nested}, while a variant of NEAR-DGD that utilizes stochastic gradient approximations only was presented in~\cite{Iakovidou2019NestedDG}. This work unifies and generalizes these methods. We summarize the related works in Table~\ref{tab:ndgd_allmethods}, denoting deterministic and random error vectors with (D) and (R), respectively.

\begin{table*}[t]
\begin{center}
\begin{tabular}{ c c c }
\hline
 \textbf{Method} & \textbf{Communication} & \textbf{Computation}\\ 
 \hline
NEAR-DGD~\cite{berahas_balancing_2019}, NEAR-DGD$^{t_c,t_g}$~\cite{berahas2020convergence} & $ \mathcal{T}_c \left[x^{j}_{i,k}\right]=x^{j}_{i,k}$ &  $\mathcal{T}_g\left[\nabla f_i \left(x^{t(k)}_{i,k}\right)\right]=\nabla f_i \left(x^{t(k)}_{i,k}\right)$ \\  
 NEAR-DGD+Q~\cite{berahas2019nested} & $ \mathcal{T}_c \left[x^{j}_{i,k}\right]=x^{j}_{i,k}+\epsilon^{j+1}_{i,k}$ (D)  & $\mathcal{T}_g\left[\nabla f_i \left(x^{t(k)}_{i,k}\right)\right]=\nabla f_i \left(x^{t(k)}_{i,k}\right)$  \\
 SG-NEAR-DGD~\cite{Iakovidou2019NestedDG} & $ \mathcal{T}_c \left[x^{j}_{i,k}\right]=x^{j}_{i,k}$ & $\mathcal{T}_g\left[\nabla f_i \left(x^{t(k)}_{i,k}\right)\right]=\nabla f_i \left(x^{t(k)}_{i,k}\right)+\zeta_{i,k+1}$ (R)  \\
S-NEAR-DGD (this paper) & $ \mathcal{T}_c \left[x^{j}_{i,k}\right]=x^{j}_{i,k}+\epsilon^{j+1}_{i,k}$ (R)  & $\mathcal{T}_g\left[\nabla f_i \left(x^{t(k)}_{i,k}\right)\right]=\nabla f_i \left(x^{t(k)}_{i,k}\right)+\zeta_{i,k+1}$ (R)  \\
 \hline
\end{tabular}
\caption{Summary of NEAR-DGD-based methods. (D) and (R) denote deterministic and random error vectors respectively.}
\label{tab:ndgd_allmethods}
\end{center}
\end{table*}

\section{Convergence Analysis}
\label{sec:analysis}
In this section, we present our theoretical results on the convergence of S-NEAR-DGD. We assume that Assumptions~\ref{assum:lip}-\ref{assum:iid} hold for the rest of this paper. We first focus on the instance of our algorithm where the number of consensus rounds is constant at every iteration, i.e., $t(k)=t$ in~\eqref{eq:near_dgd_x} for some positive integer $t>0$. We refer to this method as S-NEAR-DGD$^t$. Next, we will analyze a second variant of S-NEAR-DGD, where the number of consensus steps increases by one at every iteration, namely $t(k)=k$, for $k\geq1$. We will refer to this new version as S-NEAR-DGD$^+$. 

Before our main analysis, we introduce some additional notation and a number of preliminary results.

\subsection{Preliminaries}
We will use the notation $\mathcal{F}^j_k$  to denote the $\sigma$-algebra containing all the information generated by S-NEAR-DGD up to and including the $k$-th inexact gradient step (calculated using using $\*g_{k-1}$) and $j$ subsequent nested consensus rounds. This includes the initial point $\*x_0=\*y_0$, the vectors $\{\*x^l_\tau : 1\leq l \leq t(\tau) \text{ if }1\leq \tau<k \text{ and } 1\leq l \leq j \text{ if }\tau=k \}$, the vectors $\*y_\tau$ for $1 \leq \tau\leq k$, the vectors $\{\*q^l_\tau : 1\leq l \leq t(\tau) \text{ if }1\leq \tau<k \text{ and } 1\leq l \leq j \text{ if }\tau=k \}$ and the vectors $\*g_\tau$ for $0\leq \tau\leq k-1$. For example, $\mathcal{F}^0_k$ would denote the $\sigma$-algebra containing all the information up to and including the vector $\*y_k$ generated at the $k$-th gradient step (notice that $\mathcal{F}^0_k$ contains the inexact gradient $\*g_{k-1}$, but not $\*g_k$), while $\mathcal{F}^l_k$ would store all the information produced by S-NEAR-DGD up to and including $\*x^l_k$, generated at the $l^{th}$ consensus round after the $k$-th gradient step using $\*g_{k-1}$.

We also introduce 4 lemmas here; Lemmas ~\ref{lem:descent_f} and \ref{lem:global_L_mu} will be used  to show that the iterates generated by S-NEAR-DGD$^t$ are bounded and to characterize their distance to the solution of Problem~(\ref{eq:prob_orig}). 
Next, in Lemmas~\ref{lem:comm_error} and~\ref{lem:comp_error} we prove that the total communication and computation errors in a single iteration of  the S-NEAR-DGD$^t$ method have zero mean and bounded variance. These two error terms play a key role in our main analysis of convergence properties. 


The following lemma is adapted from~\cite[Theorem 2.1.15, Chapter 2]{nesterov_introductory_1998}.
\begin{lem}\textbf{(Gradient descent)}
\label{lem:descent_f} Let $h:\mathbb{R}^d \rightarrow \mathbb{R}$ be a $\mu$-strongly convex function with $L$-Lipschitz gradients and define $x^\star := \arg \min_x h(x)$. Then the gradient method $x_{k+1} = x_k - \alpha \nabla f\left(x_k\right)$ with steplength $\alpha < \frac{2}{\mu+L}$, generates a sequence $\{x_k\}$ such that
\begin{equation*}
\begin{split}
    \left \| x_{k+1} -x^\star \right \|^2  &\leq \left(1-\frac{2\alpha \mu L}{\mu+L}\right)\left \| x_k - x^\star\right\|^2.
\end{split}
\end{equation*}
\end{lem}

\begin{lem}\textbf{(Global convexity and smoothness)}
\label{lem:global_L_mu}
The global function $\*f:\mathbb{R}^{np}\to\mathbb{R}$, $\*f\left(\*x\right)=\frac{1}{n}\sum_{i=1}^n f_i(x_i)$ is $\mu$-strongly convex and $L$-smooth, where $\mu=\min_i \mu_i$ and $L=\max_i L_i$. In addition, the average function  $\bar f:\mathbb{R}^{p}\to\mathbb{R}$, $\bar{f}(x)=\frac{1}{n}\sum_{i=1}^n f_i(x)$ is $\mu_{\bar{f}}$-strongly convex and $L_{\bar{f}}$-smooth, where $\mu_{\bar{f}}=\frac{1}{n}\sum_{i=1}^n \mu_i$ and $L_{\bar{f}}=\frac{1}{n}\sum_{i=1}^n L_i$.
\end{lem}

\begin{proof}
This is a direct consequence of Assumptions~\ref{assum:lip} and~\ref{assum:conv}. \end{proof} 



\begin{lem}\textbf{(Bounded communication error)}
\label{lem:comm_error}
Let $\mathcal{E}^c_{t,k} := \*x^{t}_k - \*Z^{t}\*y_k$ be the total communication error at the $k$-th iteration of S-NEAR-DGD$^t$, i.e. $t(k)=t$ in~\eqref{eq:near_dgd_y} and~\eqref{eq:near_dgd_x}. Then the following relations hold for $k\geq1$,
\begin{equation*}
    \begin{split}
       \mathbb{E}_{\mathcal{T}_c}\left[\mathcal{E}^{c}_{t,k} \big| \mathcal{F}^0_k \right] = \mathbf{0}, \quad \mathbb{E}_{\mathcal{T}_c}\left[\left \|\mathcal{E}^{c}_{t,k}  \right\|^2  \big| \mathcal{F}^0_k \right]
        &\leq \frac{4n\sigma_c^2}{1-\beta^2}.
    \end{split}
\end{equation*}
\end{lem}

\begin{proof} 
Let $\tilde{\*Z}:=\*Z-I_{np}$. Setting $\*x^0_k=\*y_k$ and applying (\ref{eq:near_dgd_x}), the error $\mathcal{E}^c_{t,k}$ can be expressed as $\mathcal{E}^{c}_{t,k} = \*x^{t-1}_k + \tilde{\*Z} \*q^t_k - \*Z^t\*x^0_k$. Adding and subtracting the quantity $\sum_{j=1}^{t-1}\left(\*Z^{t-j}\*x_k^{j} \right)=\sum_{j=1}^{t-1}\left(\*Z^{t-j}\left(\*x_k^{j-1} + \tilde{\*Z} \*q^j_k \right) \right)$ (by (\ref{eq:near_dgd_x})) yields,
\begin{equation*}
    \begin{split}
    \mathcal{E}^{c}_{t,k} 
    & = \tilde{\*Z}\left(\*q^t_k-\*x^{t-1}\right)  -\sum_{j=1}^{t-2}\left(\*Z^{t-j}\*x_k^{j} \right) + \sum_{j=2}^{t-1}\left(\*Z^{t-j}\*x_k^{j-1}\right)+  \sum_{j=2}^{t-1}\left(\*Z^{t-j}\tilde{\*Z} \*q^j_k \right)+\*Z^{t-1}\tilde{\*Z}\left(\*q^1_k-\*x^0_k\right),\\
    \end{split}
\end{equation*}
where have taken the $(t-1)^{th}$ term out of $-\sum_{j=1}^{t-1}\left(\*Z^{t-j}\*x_k^{j} \right)$ and the $1^{st}$ term out of $\sum_{j=1}^{t-1}\left(\*Z^{t-j}\left(\*x_k^{j-1} + \tilde{\*Z} \*q^j_k \right) \right)$. 

 We observe that $\sum_{j=1}^{t-2}\left(\*Z^{t-j}\*x_k^{j} \right)=\sum_{j=2}^{t-1}\left(\*Z^{t-j+1}\*x_k^{j-1} \right)$, and after rearranging and combining the terms of the previous relation we obtain,
\begin{equation}
\label{eq:lem_comp_err_1}
    \begin{split}
    \mathcal{E}^{c}_{t,k} 
    & = \sum_{j=1}^{t} \left(\*Z^{t-j} \tilde{\*Z}\left(\*q^j_k - \*x^{j-1}_k \right)\right).
    \end{split}
\end{equation}

Let $d^j_k = \*q^j_k-\*x^{j-1}_k$. Noticing that $d^j_k=\left[\epsilon^j_{i,k};...;\epsilon^j_{n,k}\right]$ as defined in Assumption~\ref{assum:tc_bound}, it follows that $ \mathbb{E}_{\mathcal{T}_c}\left[d^j_k  \Big| \mathcal{F}^{j-1}_k\right] = \mathbf{0}$ for $1\leq j \leq t$.
Due to the fact that $\mathcal{F}^{0}_k \subseteq \mathcal{F}^{1}_k \subseteq ... \subseteq \mathcal{F}^{j-1}_k$, applying the tower property of conditional expectation yields,
\begin{equation}
\label{eq:comm_err_1st_mom}
    \begin{split}
         \mathbb{E}_{\mathcal{T}_c}\left[ d^j_k  \Big| \mathcal{F}^0_k\right] &= \mathbb{E}_{\mathcal{T}_c}\left[ \mathbb{E}_{\mathcal{T}_c}\left[ d^j_k \Big| \mathcal{F}^{j-1}_k\right] \Big| \mathcal{F}^0_k\right] = \mathbf{0}.
    \end{split}
\end{equation}
Combining the preceding relation with~\eqref{eq:lem_comp_err_1} and due to the linearity of expectation, we obtain $\mathbb{E}_{\mathcal{T}_c}\left[ \mathcal{E}^c_{t,k}  \big| \mathcal{F}^0_k\right] = \mathbf{0}$. This completes the first part of the proof.

Let $D^j_k = \*Z^{t-j}\tilde{\*Z} d^j_k$. By the spectral properties of $\*Z$, we have $\left\|\*Z^{t-j}\tilde{\*Z}\right\| = \max_{i>1}|\lambda_i^{t-j}||\lambda_i-1|\leq 2\beta^{t-j}$.
We thus obtain for $1\leq j \leq t$,
\begin{equation}
\label{eq:comm_err_2nd_mom}
    \begin{split}
         \mathbb{E}_{\mathcal{T}_c}&\left[ \left\|D^j_k\right\|^2 \Big| \mathcal{F}^0_k\right] \leq 4\beta^{2(t-j)} \mathbb{E}_{\mathcal{T}_c}\left[ \left\| d^j_k \right\|^2 \Big| \mathcal{F}^0_k\right]\\
         &=4\beta^{2(t-j)}\mathbb{E}_{\mathcal{T}_c}\left[ \mathbb{E}_{\mathcal{T}_c}\left[ \left\| d^j_k\right\|^2 \Big| \mathcal{F}^{j-1}_k\right] \bigg| \mathcal{F}^0_k\right] \\
         &=4\beta^{2(t-j)}\mathbb{E}_{\mathcal{T}_c}\left[  \sum_{i=1}^{n}  \mathbb{E}_{\mathcal{T}_c}\left[ \left\|\epsilon^j_{i,k} \right\|^2  \Big| \mathcal{F}^{j-1}_k\right] \bigg| \mathcal{F}^0_k\right]\\
         &\leq 4\beta^{2(t-j)}n\sigma_c^2,
    \end{split}
\end{equation}
where we derived the second inequality using the tower property of conditional expectation and applied Assumption~\ref{assum:tc_bound} to get the last inequality.

Assumption~\ref{assum:tc_bound} implies that for $i_1 \neq i_2$ and $j_1 \neq j_2$, $\epsilon^{j_1}_{i_1,k}$ and $\epsilon^{j_2}_{i_2,k}$ and by extension $d^{j_1}_k$ and $d^{j_2}_k$ are independent. Eq.~\eqref{eq:comm_err_1st_mom} then yields $\mathbb{E}_{\mathcal{T}_c}\left[\left \langle D^{j_1}_k, D^{j_2}_k \right \rangle\bigg|\mathcal{F}^0_k\right] = \mathbf{0}$. Combining this fact and linearity of expectation
yields
$\mathbb{E}_{\mathcal{T}_c} \left[\left\| \mathcal{E}^c_{t,k} \right\|^2 \big| \mathcal{F}^{0}_k \right] =  \mathbb{E}_{\mathcal{T}_c} \left[ \left\|\sum_{j=1}^{t} D^j_k\right\|^2\Bigg| \mathcal{F}^{0}_k \right] = \sum_{j=1}^{t} \mathbb{E}_{\mathcal{T}_c} \left[ \left\| D^j_k\right\|^2\Bigg| \mathcal{F}^{0}_k \right]$.
Applying~\eqref{eq:comm_err_2nd_mom} to this last relation yields,
\begin{equation*}
    \mathbb{E}_{\mathcal{T}_c} \left[\left\| \mathcal{E}^c_{t,k} \right\|^2 \big| \mathcal{F}^{0}_k \right] \leq 4 n \sigma_c^2 \sum_{j=1}^t \beta^{2(t-j)}\leq \frac{4n\sigma_c^2}{1-\beta^2},
\end{equation*}
where we used $\sum_{j=1}^t \beta^{2(t-j)}=\sum_{j=0}^{t-1} \beta^{2j} = \frac{1-\beta^{2t}}{1-\beta^2}\leq \frac{1}{1-\beta^2}$ to get the last inequality. 

\end{proof}
\begin{lem}\textbf{(Bounded computation error)}
\label{lem:comp_error}
Let $\mathcal{E}^g_{k} := \*g_{k-1} -\nabla \*f \left( \*x^t_{k-1}\right)$ be the computation error at the $k$-th iteration of S-NEAR-DGD$^t$. Then the following statements hold for all $k\geq1$, 
\begin{equation*}
      \mathbb{E}_{\mathcal{T}_g}\left[\mathcal{E}^{g}_{k} \big| \mathcal{F}^t_{k-1}\right] = \mathbf{0}, \quad \mathbb{E}_{\mathcal{T}_g}\left[\left\|\mathcal{E}^{g}_{k}\right\|^2 \big| \mathcal{F}^t_{k-1} \right] \leq n\sigma_g^2.
\end{equation*}
\end{lem}
\begin{proof}

We observe that $\mathcal{E}^{g}_{k} = \left[\zeta_{1,k};...;\zeta_{n,k}\right]$ as defined in  Assumption~\ref{assum:tg_bound}. Due to the unbiasedness of $\mathcal{T}_g\left[\cdot\right]$, we obtain
\begin{equation*}
    \mathbb{E}_{\mathcal{T}_g}\left[\mathcal{E}^{g}_{k}\big| \mathcal{F}^t_{k-1} \right] =  \mathbb{E}_{\mathcal{T}_g}\left[\*g_{k-1} -\nabla \*f \left( \*x^t_{k-1}\right)\big| \mathcal{F}^t_{k-1} \right] = \mathbf{0},
\end{equation*}
For the magnitude square of $\mathcal{E}^{g}_{k}$ we have,
\begin{equation*}
    \left\|\mathcal{E}^{g}_{k}\right\|^2 = \left\| \*g_{k-1} -\nabla \*f \left( \*x^t_{k-1}\right)\right\|^2=\sum_{i=1}^n\left\|\zeta_{i,k}\right\|^2.
\end{equation*}
Taking the expectation conditional to $\mathcal{F}^{t}_{k-1}$ on both sides of the equation above and using Assumption \ref{assum:tg_bound} establishes the desired results.
\end{proof}

We are now ready to proceed with our main analysis of the convergence properties of S-NEAR-DGD.

\subsection{Main Analysis}
For simplicity, from this point on we will use the notation $\mathbb{E}\left[ \cdot \right]$ to denote the expected value taken over the randomness of both $\mathcal{T}_c$ and $\mathcal{T}_g$. 
We begin our convergence analysis by proving that the iterates generated by S-NEAR-DGD$^t$ are bounded in expectation in Lemma~\ref{lem:bounded_iterates}. Next, we demonstrate that the distance between the local iterates produced by our method and their average is bounded in Lemma~\ref{lem:bounded_variance}. In Lemma~\ref{lem:descent_avg_iter}, we prove an intermediate result stating that the distance between the average iterates of S-NEAR-DGD$^t$ and the optimal solution is bounded. We then use this result to show the linear of convergence of S-NEAR-DGD$^t$ to a neighborhood of the optimal solution in Theorem~\ref{thm:bounded_dist_min}, and we characterize the size of this error neighborhood in terms of network and problem related quantities and the precision of the stochastic gradients and the noisy communication channel. We prove convergence to a neighborhood of the optimal solution for the local iterates of S-NEAR-DGD$^t$ in Corollary~\ref{cor:local_dist}. We conclude our analysis by proving that the average iterates of S-NEAR-DGD$^+$ converge with geometric rate to an improved error neighborhood compared to S-NEAR-DGD$^t$ in Theorem~\ref{thm:near_dgd_plus}.



\begin{lem}\textbf{(Bounded iterates)}
\label{lem:bounded_iterates} Let $\*x_k$ and $\*y_k$ be the iterates generated by S-NEAR-DGD$^t$ ($t(k)=t$ in Eq. (\ref{eq:near_dgd_x}) and (\ref{eq:near_dgd_y})) starting from initial point $\*y_0=\*x_0 \in \mathbb{R}^{np}$ and let the steplength $\alpha$ satisfy
\begin{equation*}
    \alpha < \frac{2}{\mu + L},
\end{equation*}
where $\mu=\min_i \mu_i$ and $L=\max_i L_i$.

Then $\*x_k$ and $\*y_k$ are bounded in expectation for $k\geq1$, i.e.,
\begin{equation*}
    \begin{split}
    \mathbb{E}\left[\left\| \*y_{k}\right\|^2 \right]\leq D + 
      \frac{(1+\kappa)^2n\sigma_g^2}{2L^2}+ \frac{2(1+\kappa)^2 n\sigma_c^2}{\alpha\left(1-\beta^2\right)L^2},
    \end{split}
\end{equation*}
\begin{equation*}
    \begin{split}
        \mathbb{E}\left[\left\| \*x^t_k \right\|^2 \right] \leq D + 
      \frac{(1+\kappa)^2n\sigma_g^2}{2L^2}+ \frac{2(1+\kappa)^2 n\sigma_c^2}{\alpha\left(1-\beta^2\right)L^2}  +\frac{4n\sigma_c^2}{1-\beta^2},
    \end{split}
\end{equation*}
where $D=2 \mathbb{E}\left[
       \left\|\*y_0 - \*u^\star \right\|^2 \right]+2\left(1+4\nu^{-3}\right)\left\|\*u^\star\right\|^2$, $\*u^\star=[u_1^\star;u_2^\star;...;u_n^\star] \in \mathbb{R}^{np}$,  $u_i^\star = \arg \min_x f_i(x)$, $\nu=\frac{2\alpha\mu L}{\mu+L}$ and $\kappa=L/\mu$ is the condition number of Problem~\eqref{eq:consensus_prob}.
\end{lem}

\begin{proof}
Consider,
\begin{equation*}
\label{eq:lem_bounded_xy_1}
    \begin{split}
        \left\| \*y_{k+1} - \*u^\star \right\|^2 &=  \left\|\*x_k^{t} - \alpha \*g_k - \*u^\star \right\|^2\\
        &= \left\|\*x_k^{t} - \alpha \nabla \*f(\*x_k^{t}) - \*u^\star  - \alpha  \mathcal{E}^g_{k+1}  \right\|^2 \\
        &= \left\|\*x_k^{t} - \alpha \nabla \*f(\*x_k^{t}) - \*u^\star \right\|^2 + \alpha^2\left\| \mathcal{E}^g_{k+1} \right\|^2 - 2\alpha \left  \langle \*x_k^{t} - \alpha \nabla \*f(\*x_k^{t}) - \*u^\star, \mathcal{E}^g_{k+1}\right \rangle,
    \end{split}
\end{equation*}
where we used (\ref{eq:near_dgd_y}) to get the first equality and added and subtracted $\alpha\nabla\*f\left(\*x^t_k\right)$ and applied the computation error definition $\mathcal{E}^g_{k+1}:=\*g_k - \nabla \*f\left(\*x^t_k\right)$ to obtain the second equality.

Taking the expectation conditional to $\mathcal{F}^t_k$ on both sides of
the inequality above
and applying Lemma~\ref{lem:comp_error} yields,
\begin{equation}
\label{eq:lem_bounded_xy_2}
    \begin{split}
        \mathbb{E}\left[\left\| \*y_{k+1} - \*u^\star \right\|^2 \Big| \mathcal{F}^t_k \right]&\leq \left\|\*x_k^{t} - \alpha \nabla \*f(\*x_k^{t}) - \*u^\star \right\|^2 + \alpha^2 n \sigma_g^2.
    \end{split}
\end{equation}
For the first term on the right-hand side of (\ref{eq:lem_bounded_xy_2}), after combining Lemma~\ref{lem:descent_f}
with Lemma~\ref{lem:global_L_mu} and due to $\alpha < \frac{2}{\mu+L}$ we acquire,
\begin{equation*}
\label{eq:lem_bounded_xy_3}
   \left\|\*x_k^{t} - \alpha \nabla \*f(\*x_k^{t}) - \*u^\star  \right\|^2 \leq (1-\nu)\left\|\*x_k^{t}-\*u^\star\right\|^2,
\end{equation*}
where $\nu=\frac{2\alpha \mu L}{\mu + L}=\frac{2\alpha L}{1+\kappa}<1$.

Expanding the term on the right hand side of the above relation yields,
\begin{equation*}
\label{eq:lem_bounded_xy_4}
    \begin{split}
      \left\|\*x_k^{t}-\*u^\star\right\|^2 &= \left\|\mathcal{E}^{c}_{t,k} + \*Z^t\*y_k -\*Z^t\*u^\star + \*Z^t\*u^\star - \*u^\star\right\|^2\\
      &= \left\|\mathcal{E}^{c}_{t,k}\right\|^2 + \left\|\*Z^t\left(\*y_k - \*u^\star\right) - \left(I-\*Z^t\right)\*u^\star\right\|^2 + 2\left \langle \mathcal{E}^{c}_{t,k}, \*Z^t\*y_k - \*u^\star  \right \rangle \\
      &\leq \left\|\mathcal{E}^{c}_{t,k}\right\|^2 + \left(1+\nu\right)\left\|\*Z^t\left(\*y_k - \*u^\star\right)\right\|^2 +\left(1+\nu^{-1}\right)\left\| \left(I-\*Z^t\right)\*u^\star\right\|^2 + 2\left \langle \mathcal{E}^{c}_{t,k}, \*Z^t\*y_k - \*u^\star \right \rangle\\
      &\leq \left\|\mathcal{E}^{c}_{t,k}\right\|^2 + \left(1+\nu\right)\left\|\*y_k - \*u^\star\right\|^2 +4\left(1+\nu^{-1}\right)\left\|\*u^\star\right\|^2 + 2\left \langle \mathcal{E}^{c}_{t,k}, \*Z^t\*y_k - \*u^\star  \right \rangle,
    \end{split}
\end{equation*}
where we added and subtracted the quantities $\*Z^t\*y_k$ and $\*Z^t\*u^\star$ and applied the communication error definition $\mathcal{E}^{c}_{t,k}:=\*x^t_k - \*Z^t\*y_k$ to get the first equality. We used the standard inequality $\pm2\langle a,b\rangle \leq c\|a\|^2 + c^{-1}\|b\|^2$ that holds for any two vectors $a,b$ and positive constant $c>0$ to obtain the first inequality.
Finally, we derived the last inequality using the relations $\left\|\*Z^t \right\|=1$ and $\left\|I-\*Z^t \right\|<2$ that hold due to Assumption~\ref{assum:consensus_mat}.

Due to the fact that $\mathcal{F}^0_k \subseteq \mathcal{F}^t_k$, combining the preceding three relations and taking the expectation conditional on $\mathcal{F}^0_k$ on both sides of~\eqref{eq:lem_bounded_xy_2} yields,
 \begin{equation*}
     \begin{split}
        \mathbb{E}\Big[\big\| \*y_{k+1} - \*u^\star \big\|^2 \Big| \mathcal{F}^0_k \Big] &\leq \left(1-\nu^2\right) 
        \left\|\*y_k - \*u^\star \right\|^2 + 
        \alpha^2n\sigma_g^2 \\
        &\quad+  (1-\nu)\mathbb{E}\left[\left\|\mathcal{E}^c_{t,k}\right\|^2\big|\mathcal{F}^0_k\right]+ 4\nu^{-1}\left(1-\nu^2\right)\left\|\*u^\star\right\|^2+  2(1-\nu)\mathbb{E}\left[\left \langle \mathcal{E}^{c}_{t,k}, \*Z^t\*y_k - \*u^\star  \right \rangle \big|\mathcal{F}^0_k\right]\\
        &\leq \left(1-\nu^2\right) 
        \left\|\*y_k - \*u^\star \right\|^2 + 
        \alpha^2n\sigma_g^2 +  (1-\nu)\frac{4n\sigma_c^2}{1-\beta^2}+ 4\nu^{-1}\left(1-\nu^2\right)\left\|\*u^\star\right\|^2,
    \end{split}
 \end{equation*}
 where we applied Lemma~\ref{lem:comm_error} to get the last inequality.

Taking the total expectation on both sides of the relation above and applying recursively over iterations $0, 1, \ldots, k$ yields,
\begin{equation*}
\label{eq:lem1_1}
    \begin{split}
        \mathbb{E}\Big[\big\| \*y_{k} - \*u^\star \big\|^2\Big]
        &\leq \left(1-\nu^2\right)^k \mathbb{E}\left[
       \left\|\*y_0 - \*u^\star \right\|^2 \right]+
        \alpha^2\nu^{-2}n\sigma_g^2 +  (\nu^{-2}-\nu^{-1})\frac{4n\sigma_c^2}{1-\beta^2} + 4\left(\nu^{-3}-\nu^{-1}\right)\left\|\*u^\star\right\|^2\\
        &\leq \mathbb{E}\left[
       \left\|\*y_0 - \*u^\star \right\|^2 \right]+
        \alpha^2\nu^{-2}n\sigma_g^2 +  \frac{4\nu^{-2}n\sigma_c^2}{1-\beta^2} + 4\nu^{-3}\left\|\*u^\star\right\|^2,
        \end{split}
\end{equation*}
where we used $\sum_{h=0}^{k-1}\left(1-\nu^2\right)^h\leq \nu^{-2}$ to get the first inequality 
and $\nu > 0$ to get the second inequality.
 
Moreover, the statement $\left\| \*y_{k}\right\|^2 = \left\| \*y_{k} - \*u^\star + \*u^\star \right\|^2\leq 2\left\| \*y_{k} - \*u^\star \right\|^2+2\left\| \*u^\star \right\|^2$ trivially holds.
Taking the total expectation on both sides of this relation, yields,
\begin{equation}
\label{eq:bound_yk}
\begin{split}
    \mathbb{E}\Big[\big\| \*y_{k}\big\|^2 \Big]& \leq 2\mathbb{E}\left[\left\| \*y_{k} - \*u^\star \right\|^2\right]+2\left\| \*u^\star \right\|^2\\
    &\leq 2 \mathbb{E}\left[
       \left\|\*y_0 - \*u^\star \right\|^2 \right]+ 
       2 \alpha^2\nu^{-2}n\sigma_g^2 + \frac{8\nu^{-2}n\sigma_c^2}{1-\beta^2}+ 2\left(1+4\nu^{-3}\right)\left\|\*u^\star\right\|^2.
    \end{split}
\end{equation}
Applying the definitions of $D$ and $\kappa$ to (\ref{eq:bound_yk}) yields the first result of this lemma.

Finally, the following statement also holds 
\begin{equation*}
    \begin{split}
        \left\| \*x^t_k \right\|^2 &= \left\| \mathcal{E}^{c}_{t,k} + \*Z^t\*y_k\right\|^2\\
        &= \left\| \mathcal{E}^{c}_{t,k}\right\|^2 +\left\| \*Z^t\*y_k\right\|^2 + 2\left \langle\mathcal{E}^{c}_{t,k}, \*Z^t\*y_k \right \rangle\\
         &\leq \left\| \mathcal{E}^{c}_{t,k}\right\|^2 +\left\| \*y_k\right\|^2  + 2\left \langle\mathcal{E}^{c}_{t,k}, \*Z^t\*y_k \right \rangle,
    \end{split}
\end{equation*}
where we used the non-expansiveness of $\*Z$ for the last inequality.

Taking the expectation conditional  on $\mathcal{F}^0_k$ on both sides of the preceding relation and applying Lemma \ref{lem:comm_error} yields,
\begin{equation*}
\label{eq:bound_xk}
    \begin{split}
        \mathbb{E}\left[\left\| \*x^t_k \right\|^2 \Big | \mathcal{F}^0_k \right]  &\leq  \frac{4n\sigma_c^2}{1-\beta^2} +\left\| \*y_k\right\|^2.
    \end{split}
\end{equation*}
Taking the total expectation on both sides of the relation above, applying (\ref{eq:bound_yk})  and the definitions of $D$, $\nu$ and $\kappa$ concludes this proof.

\end{proof}

We will now use the preceding lemma to prove that the distance between the local and the average iterates generated by S-NEAR-DGD$^t$ is bounded. This distance can be interpreted as a measure of consensus violation, with small values indicating small disagreement between nodes.

\begin{lem}
\textbf{(Bounded distance to average)}
\label{lem:bounded_variance}
Let $x_{i,k}^t$ and $y_{i,k}$ be the local iterates produced by S-NEAR-DGD$^t$ at node $i$ and iteration $k$ and let $\bar{x}_k^t:=\sum_{i=1}^n x^t_{i,k}$ and $\bar{y}_k:=\sum_{i=1}^n y_{i,k}$ denote the average iterates across all nodes. Then the distance between the local and average iterates is bounded in expectation for all $i=1,...,n$ and $k=1,2,...$, namely,
\begin{equation*}
\begin{split}
    &\mathbb{E}\Big[\big\|x^t_{i,k} - \bar{x}^t_k \big\|^2\Big] \leq \mathbb{E}\left[\left\|\*x^t_k - \*M \*x^t_k \right\|^2\right] \leq \beta^{2t}D + 
      \frac{\beta^{2t}(1+\kappa)^2n\sigma_g^2}{2L^2}+ \frac{2\beta^{2t}(1+\kappa)^2n\sigma_c^2}{\alpha^2\left(1-\beta^2\right) L^2} +\frac{4n\sigma_c^2}{1-\beta^2},
     \end{split}
\end{equation*}
and
\begin{equation*}
\begin{split}
    &\mathbb{E} \Big[\big\|y_{i,k} - \bar{y}_k \big\|^2 \Big]\leq \mathbb{E}\left[ \left\|\*y_k - \*M \*y_k \right\|^2 \right]\leq D + 
      \frac{(1+\kappa)^2n\sigma_g^2}{2L^2}+ \frac{2(1+\kappa)^2n\sigma_c^2}{\alpha^2\left(1-\beta^2\right) L^2},
    \end{split}
\end{equation*}
where 
$\*M = \left( \frac{1_n 1_n^T}{n} \otimes I_p \right) \in \mathbb{R}^{np}$ is the averaging matrix, constant $D$ is defined in Lemma~\ref{lem:bounded_iterates}, $\kappa=L/\mu$ is the condition number of Problem~\eqref{eq:consensus_prob}, $L=\max_i L_i$ and $\mu = \min_i L_i$. 
 
\end{lem}

\begin{proof}

Observing that $\sum_{i=1}^n \left\|x^t_{i,k}-\bar{x}^t_k \right\|^2 = \left\| \*x^t_k - \*M \*x^t_k \right\|^2$, we obtain,
\begin{equation}
\label{eq:lem_bounded_diff_x}
    \left\|x^t_{i,k}-\bar{x}^t_k \right\|^2 \leq \left\| \*x^t_k - \*M \*x^t_k \right\|^2,\quad i=1,...,n.
\end{equation}
We can bound the right-hand side of (\ref{eq:lem_bounded_diff_x}) as
\begin{equation}
\label{eq:lem2_x1}
    \begin{split}
        \left\|\*x^t_k - \*M \*x_k\right\|^2 &=  \left\|\mathcal{E}^{c}_{t,k} + \*Z^t\*y_k- \*M \*x_k - \*M\*y_k + \*M\*y_k \right\|^2\\
        &=\left\|\mathcal{E}^{c}_{t,k} + \left(\*Z^t-\*M\right)\*y_k- \*M\*x_k +\*M \*Z^t\*y_k \right\|^2\\
        &=\left\|\left(I-\*M\right)\mathcal{E}^{c}_{t,k} \right\|^2+\left\| \left(\*Z^t-\*M\right)\*y_k \right\|^2  + 2\left \langle\left(I-\*M\right)\mathcal{E}^{c}_{t,k} ,\left(\*Z^t-\*M\right)\*y_k  \right \rangle\\
        &\leq \left\|\mathcal{E}^{c}_{t,k} \right\|^2+\beta^{2t}\left\|\*y_k \right\|^2 + 2\left \langle\left(I-\*M\right)\mathcal{E}^{c}_{t,k} ,\left(\*Z^t-\*M\right)\*y_k  \right \rangle,
    \end{split}
\end{equation}
where we applied the definition of the communication error $\mathcal{E}^c_{t,k}$ of Lemma~\ref{lem:comm_error}
and added and subtracted $\*M\*y_k$ to obtain the first equality. We used the fact that $\*M \*Z^t = \*M$ to get the second equality.
We derive the last inequality from Cauchy-Schwarz and the spectral properties of $\*Z^t = \*W^t \otimes I_p$ and $\*M = \left(\frac{1_n 1_n^T}{n}\right) \otimes I_p$; both $\*W^t$ and $\frac{1_n 1_n^T}{n}$ have a maximum eigenvalue at $1$ associated with the eigenvector $1_n$, implying that the null space of $\*W^t - \frac{1_n 1_n^T}{n}$ is parallel to $1_n$ and $\left\|\*Z^t-\*M\right\|=\left\|\*W^t-\frac{1_n 1_n^T}{n}\right\|=\beta^t$. 


Taking the expectation conditional to $\mathcal{F}^0_k$ on both sides of~(\ref{eq:lem2_x1}) and applying Lemma~\ref{lem:comm_error} yields,
\begin{equation*}
    \begin{split}
       \mathbb{E}\left[ \left\|\*x^t_k - \*M \*x_k\right\|^2 \Big | \mathcal{F}^0_k \right] &\leq \frac{4n\sigma_c^2}{1-\beta^2} +\beta^{2t}\left\|\*y_k \right\|^2.
    \end{split}
\end{equation*}
Taking the total expectation on both sides and applying Lemma~\ref{lem:bounded_iterates} yields the first result of this lemma.

Similarly, the following inequality holds for the $\*y_k$ iterates,
\begin{equation}
\label{eq:lem_bounded_diff_y}
    \left\|y_{i,k}-\bar{y}_k \right\|^2 \leq \left\| \*y_k - \*M \*y_k \right\|^2,\quad i=1,...,n.
\end{equation}
For the right-hand side of (\ref{eq:lem_bounded_diff_y}), we have,
\begin{equation*}
    \begin{split}
        \left\|\*y_k - \*M\*y_k \right\|^2 &= \left\| \left(I-\*M\right)\*y_k\right\|^2\leq \left\| \*y_k \right\|^2,
\end{split}
\end{equation*}
where we have used the fact that $\left\|I-\*M\right\|=1$.

Taking the total expectation on both sides and applying Lemma~\ref{lem:bounded_iterates} concludes this proof.

\end{proof}
The bounds established in Lemma~\ref{lem:bounded_variance} indicate that there are at least three factors preventing the local iterates produced by S-NEAR-DGD$^t$ from reaching consensus: errors related to network connectivity, represented by $\beta$, and errors caused by the inexact computation process and the noisy communication channel associated with the constants $\sigma_g$ and $\sigma_c$ respectively.

Before presenting our main theorem, we state one more intermediate result on the distance of the average $\bar{y}_k$ iterates to the solution of Problem~(\ref{eq:consensus_prob}). 

\begin{lem}
\textbf{(Bounded distance to minimum)}
\label{lem:descent_avg_iter}
Let $\bar{y}_k:=\frac{1}{n}\sum_{i=1}^n y_{i,k}$ denote the average of the local $y_{i,k}$ iterates generated by S-NEAR-DGD$^t$ under steplength $\alpha$ satisfying
\begin{equation*}
    \alpha < \frac{2}{\mu_{\bar{f}}+L_{\bar{f}}},
\end{equation*}
where $\mu_{\bar{f}}=\frac{1}{n}\sum_{i=1}^n\mu_i$ and $L_{\bar{f}}=\frac{1}{n}\sum_{i=1}^n L_i$.

Then the following inequality holds for $k=1,2,...$
\begin{equation*}
\label{eq:dist_y_bar_x_star}
    \begin{split}
       \mathbb{E}\left[ \left\| \bar{y}_{k+1} - x^\star\right\|^2 \Big | \mathcal{F}^t_k \right]
      &\leq  \rho  \left\| \bar{x}^t_{k} - x^\star\right\|^2 + \frac{\alpha^2\sigma_g^2}{n} +  \frac{\alpha \rho L^2 \Delta_{\*x}}{n\gamma_{\bar{f}}},
    \end{split}
\end{equation*}
where $x^\star = \arg\min_x f(x)$, $\rho=1-\alpha\gamma_{\bar{f}}$, $\gamma_{\bar{f}}=\frac{\mu_{\bar{f}}L_{\bar{f}}}{\mu_{\bar{f}}+L_{\bar{f}}}$, $L=\max_i L_i$, $\Delta_{\*x}=\left\|\*x^t_k - \*M \*x^t_k\right\|^2$ and $\*M = \left( \frac{1_n 1_n^T}{n} \otimes I_p \right) \in \mathbb{R}^{np}$ is the averaging matrix.
\end{lem}

\begin{proof}
Applying ($\ref{eq:near_dgd_y_avg}$) to $(k+1)^{th}$ iteration we obtain,
\begin{equation*}
\begin{split}
\bar{y}_{k+1} = \bar{x}_k^{t} - \alpha \overline{g}_k,
    \end{split}
\end{equation*}
where $\bar{g}_k = \frac{1}{n}\sum_{i=1}^n g_{i,k} =  \frac{1}{n}\sum_{i=1}^n \left(\zeta_{i,k+1}+ \nabla f_i\left(x^t_{i,k}\right)\right)$. 
Let $h_k = \frac{1}{n}\sum_{i=1}^n \nabla f_i\left(x^t_{i,k}\right)$. Adding and subtracting $\alpha h_k$ to the right-hand side of the preceding relation and taking the square norm on both sides yields,
\begin{equation*}
\begin{split}
    \left\|\bar{y}_{k+1} - x^\star\right\|^2 &= \left\|\bar{x}^t_{k} -\alpha h_k - x^\star\right\|^2 + \alpha^2 \left\| h_k  -\bar{g}_k \right\|^2  + 2\alpha \left \langle  \bar{x}^t_{k} -\alpha h_k - x^\star, h_k -\bar{g}_k \right \rangle\\
    &= \left\|\bar{x}^t_{k} -\alpha h_k - x^\star\right\|^2 + \frac{\alpha^2}{n^2} \left\| \sum_{i=1}^n \zeta_{i,k+1} \right\|^2 - \frac{2\alpha }{n}\sum_{i=1}^n\left \langle  \bar{x}^t_{k} -\alpha h_k - x^\star, \zeta_{i,k+1} \right \rangle.
    \end{split}
\end{equation*}
Moreover, let $\tilde{\rho}=\frac{\alpha\gamma_{\bar{f}}}{1-2\alpha\gamma_{\bar{f}}} > 0$. We can re-write the first term on the right-hand side of the inequality above as,
\begin{equation*}
    \begin{split}
        \Big\|\bar{x}^t_{k} -\alpha h_k - x^\star\Big\|^2 &\leq \left(1 + \tilde{\rho}\right) \left\|\bar{x}^t_{k} -\alpha\nabla \bar{f}\left(\bar{x}^t_k\right) - x^\star\right\|^2 + \alpha^2 \left(1+\tilde{\rho}^{-1}\right) \left\|h_k -\nabla \bar{f}\left(\bar{x}^t_k\right)\right\|^2\\
        &\leq \left(1 -\alpha \gamma_{\bar{f}}\right) \left\|\bar{x}^t_{k} - x^\star\right\|^2 + \alpha^2 \left(1+\tilde{\rho}^{-1}\right) \left\|h_k - \nabla \bar{f}\left(\bar{x}^t_k\right)\right\|^2,
    \end{split}
\end{equation*}
where we added and subtracted the quantity $\alpha \nabla \bar{f}\left(\bar{x}^t_k\right)$ and used the relation $\pm 2\langle a,b\rangle \leq  c\|a\|^2+c^{-1}\|b\|^2$ that holds for any two vectors $a,b$ and positive constant $c$ to obtain the first inequality. We derive the second inequality after combining Lemmas~\ref{lem:global_L_mu} and~\ref{lem:descent_f} that hold due to $\alpha < \frac{2}{\mu_{\bar{f}}+L_{\bar{f}}}$ and $x^\star=\arg \min_x \bar{f}(x)$.

We notice that $\mathbb{E}\left[ \zeta_{i,k+1} \big| \mathcal{F}^t_k\right]=\mathbf{0}$ and that $\mathbb{E}\left[ \left\|\sum_{i=1}^n \zeta_{i,k+1}\right\|^2 \big|\mathcal{F}^t_k\right] =  \mathbb{E}\left[\sum_{i=1}^n \left\|\zeta_{i,k+1}\right\|^2\big| \mathcal{F}^t_k\right]
   +\mathbb{E}\left[ \sum_{i_1\neq i_2} \left\langle \zeta_{i_1,k+1},\zeta_{i_2,k+1} \right\rangle \big| \mathcal{F}^t_k\right] \leq n\sigma_g^2$ due to Assumption~\ref{assum:tg_bound} and the linearity of expectation. Combining all of the preceding relations and taking the expectation conditional on $\mathcal{F}^t_k$, yields,
\begin{equation}
\label{eq:lem_bounded_min100}
    \begin{split}
       \mathbb{E} \Big[\big\| &\bar{y}_{k+1}  - x^\star \big\|^2 \big| \mathcal{F}^t_k \Big] = \left(1-\alpha \gamma_{\bar{f}}\right)\left\|\bar{x}^t_k - x^\star\right\|+ \alpha^2 \left(1+\tilde{\rho}^{-1}\right) \left\|h_k -  \nabla \bar{f}\left(\bar{x}^t_k\right)\right\|^2+ \frac{\alpha^2 \sigma_g^2}{n}
    \end{split}
\end{equation}

Finally, for any set of vectors $v_i \in \mathbb{R}^p$, $i=1,...,n$ we have $  \left\|\sum_{i=1}^n v_i \right\|^2 = \sum_{h=1}^p \left(\sum_{i=1}^n \left[v_i \right]_h \right)^2
        \leq n \sum_{h=1}^p \sum_{i=1}^n \left[v_i \right]_h^2 
         = n \sum_{i=1}^n \left\|v_i\right\|^2,$
where we used the fact that $\pm 2 a  b \leq a^2 + b^2$ for any pair of scalars $a,b$ to get the first inequality and reversed the order of summation to get the last equality.
We can use this result to obtain,
\begin{equation*}
\label{eq:lem_bounded_min103}
\begin{split}
    \Big\|h_k -  \nabla \bar{f}\left(\bar{x}^t_k\right)\Big\|^2 &=\frac{1}{n^2}\left\|\sum_{i=1}^n \left(\nabla f_i \left(x^t_{i,k}\right)-\nabla f_i\left(\bar{x}^t_k\right)\right)\right\|^2\\
    &\leq \frac{n}{n^2} \sum_{i=1}^n \left\|\nabla f_i \left(x^t_{i,k}\right)-\nabla f_i\left(\bar{x}^t_k\right)\right\|^2\\
    & \leq \frac{L^2 }{n}  \sum_{i=1}^n \left\|x^t_{i,k}-\bar{x}^t_k\right\|^2
\\
&= \frac{L^2}{n}   \left\|\*x^t_k - \*M \*x^t_k\right\|^2,
    \end{split}
\end{equation*}
where we used Assumption~\ref{assum:lip} to get the second inequality.

Substituting the immediately previous relation in (\ref{eq:lem_bounded_min100}), observing $1+\tilde{\rho}^{-1}=\left(1-\alpha\gamma_{\bar{f}}\right)/\alpha\gamma_{\bar{f}}$ and applying the definition of $\rho$ yields the final result.

\end{proof}

We have now obtained all necessary results to prove the convergence of S-NEAR-DGD$^t$ to a neighborhood of the optimal solution in the next theorem. 

\begin{thm}
\textbf{(Convergence of S-NEAR-DGD$^t$)}
\label{thm:bounded_dist_min}
Let $\bar{x}^t_k:=\frac{1}{n}\sum_{i=1}^n x^t_{i,k}$ denote the average of the local $x^t_{i,k}$ iterates generated by S-NEAR-DGD$^t$ from initial point $\*y_0$
and let the steplength $\alpha$ satisfy,
\begin{equation*}
    \alpha < \min \left \{\frac{2}{\mu+L},\frac{2}{\mu_{\bar{f}}+L_{\bar{f}}} \right\},
\end{equation*}
where $\mu=\min_i L_i$, $L = \max_i L_i$, $\mu_{\bar{f}}=\frac{1}{n}\sum_{i=1}^n\mu_i$ and $L_{\bar{f}}=\frac{1}{n}\sum_{i=1}^n L_i$.

Then the distance of $\bar{x}^t_k$ to the optimal solution $x^\star$ of Problem~(\ref{eq:consensus_prob}) is bounded in expectation for $k=1,2,...$,
\begin{equation}
\label{eq:thm_xk1}
    \begin{split}
    &\mathbb{E}\left[\left\|\bar{x}^t_{k+1} - x^\star\right\|^2  \right]
    \leq \rho \mathbb{E}\left[ \left\| \bar{x}^t_{k} - x^\star\right\|^2\right] + \frac{\alpha\beta^{2t}\rho L^2D}{n\gamma_{\bar{f}}} + \frac{\alpha^2\sigma_g^2  }{n} +  \frac{\alpha\beta^{2t}\left(1+\kappa\right)^2\rho\sigma_g^2}{2\gamma_{\bar{f}}} + \frac{4\alpha\rho L^2\sigma_c^2}{\left(1-\beta^2\right)\gamma_{\bar{f}}} +  \frac{2\beta^{2t}\left(1+\kappa\right)^{2}\rho \sigma_c^2}{\alpha\left(1-\beta^2\right)\gamma_{\bar{f}}},
    \end{split}
\end{equation}
and
\begin{equation}
\label{eq:thm_xk2}
    \begin{split}
    &\mathbb{E}\left[\left\|\bar{x}^t_{k} - x^\star\right\|^2  \right]
    \leq  \rho^k \mathbb{E}\left[ \left\| \bar{x}_0 - x^\star\right\|^2\right] + \frac{\beta^{2t}\rho L^2D}{n\gamma_{\bar{f}}^2} + \frac{\alpha \sigma_g^2  }{n\gamma_{\bar{f}}}  +  \frac{\beta^{2t} \left(1+\kappa \right)^2\rho\sigma_g^2}{2\gamma_{\bar{f}}^2}+ \frac{4\rho L^2\sigma_c^2}{\left(1-\beta^2\right)\gamma_{\bar{f}}^2} +  \frac{2\beta^{2t}\left(1+\kappa\right)^2\rho \sigma_c^2}{\alpha^2\left(1-\beta^2\right)\gamma_{\bar{f}}^2},
    \end{split}
\end{equation}
where $\bar{x}_0=\frac{1}{n}\sum_{i=1}^n y_{i,0}$, $\rho=1-\alpha\gamma_{\bar{f}}$, $\gamma_{\bar{f}}=\frac{\mu_{\bar{f}}L_{\bar{f}}}{\mu_{\bar{f}}+L_{\bar{f}}}$, $\kappa=L/\mu$ is the condition number of Problem~\eqref{eq:consensus_prob} and the constant $D$ is defined in Lemma~\ref{lem:bounded_iterates}.
\end{thm}

\begin{proof}
Applying ($\ref{eq:near_dgd_x_avg}$) to the $(k+1)^{th}$ iteration yields,
\begin{equation*}
\begin{split}
    \bar{x}^j_{k+1} &=  \bar{x}^{j-1}_{k+1}\text{, }j=1,...,t,
    \end{split}
\end{equation*}
which in turn implies that $\bar{x}^j_{k+1}=\bar{y}_{k+1}$ for $j=1,...,t$.

Hence, the relation $\left\|\bar{x}^t_{k+1} - x^\star\right\|^2 =\left\| \bar{y}_{k+1}- x^\star\right\|^2$ holds.
Taking the expectation conditional to $\mathcal{F}^t_{k}$  on both sides of this equality and applying Lemma~\ref{lem:descent_avg_iter} yields,
 \begin{equation*}
    \begin{split}
    \mathbb{E}\left[\left\|\bar{x}^t_{k+1} - x^\star\right\|^2 \Big| \mathcal{F}^t_{k}\right] 
    &\leq   \rho  \left\| \bar{x}^t_{k} - x^\star\right\|^2 + \frac{\alpha^2\sigma_g^2}{n}+  \frac{\alpha \rho L^2 \Delta_{\*x}}{n\gamma_{\bar{f}}},
    \end{split}
\end{equation*}
where $\Delta_{\*x}=\left\|\*x^t_k - \*M \*x^t_k\right\|^2$.

Taking the total expectation on both sides of the relation above and applying Lemma~\ref{lem:bounded_variance} yields,
\begin{equation}
\label{eq:lem_bounded_min6}
    \begin{split}
    &\mathbb{E}\left[\left\|\bar{x}^t_{k+1} - x^\star\right\|^2  \right]
    \leq \rho \mathbb{E}\left[ \left\| \bar{x}^t_{k} - x^\star\right\|^2\right] + \frac{\alpha\beta^{2t}\rho L^2D}{n\gamma_{\bar{f}}}  + \frac{\alpha^2 \sigma_g^2  }{n} +  \frac{\alpha\beta^{2t}(1+\kappa)^{2}\rho \sigma_g^2}{2\gamma_{\bar{f}}} + \frac{4\alpha\rho L^2\sigma_c^2}{\gamma_{\bar{f}}\left(1-\beta^2\right)} +  \frac{2\beta^{2t}(1+\kappa)^{2}\rho \sigma_c^2}{\alpha\left(1-\beta^2\right)\gamma_{\bar{f}}}.
    \end{split}
\end{equation}
We notice that $\rho<1$ and after applying \eqref{eq:lem_bounded_min6} recursively and then using the bound $\sum_{h=0}^{k-1} \rho^h \leq (1-\rho)^{-1}$ we obtain,
\begin{equation*}
    \begin{split}
    \mathbb{E}\left[\left\|\bar{x}^t_{k} - x^\star\right\|^2  \right]
    &\leq \rho^k \mathbb{E}\left[ \left\| \bar{x}_0- x^\star\right\|^2\right] + \frac{\alpha\beta^{2t}\rho L^2D}{n\gamma_{\bar{f}}(1-\rho)} + \frac{\alpha^2 \sigma_g^2  }{n(1-\rho)} \\
    &\quad +  \frac{\alpha\beta^{2t}(1+\kappa)^{2}\rho \sigma_g^2}{2\gamma_{\bar{f}}(1-\rho)} + \frac{4\alpha\rho L^2\sigma_c^2}{\gamma_{\bar{f}}\left(1-\beta^2\right)(1-\rho)} +  \frac{2\beta^{2t}(1+\kappa)^{2}\rho \sigma_c^2}{\alpha\left(1-\beta^2\right)\gamma_{\bar{f}}(1-\rho)}.
    \end{split}
\end{equation*}
Applying the definition of $\rho$ completes the proof.

\end{proof}

Theorem~\ref{thm:bounded_dist_min} indicates that the average iterates of S-NEAR-DGD$^t$ converge in expectation to a neighborhood of the optimal solution $x^\star$ of Problem~(\ref{eq:prob_orig}). We have quantified the dependence of this neighborhood on the connectivity of the network and the errors due to imperfect computation and communication through the terms containing the quantities $\beta$, $\sigma_g$ and $\sigma_c$, respectively. 
We observe that the $2^{nd}$ and $3^{rd}$ error terms in~\eqref{eq:thm_xk2} scale favorably with the number of nodes $n$, yielding a variance reduction effect proportional to network size. Our bounds indicate that higher values of the steplength $\alpha$ yield faster convergence rates $\rho$. On the other hand, $\alpha$ has a mixed effect on the size of the error neighborhood; the $2^{nd}$ term in~\eqref{eq:thm_xk2} is associated with inexact computation and increases with $\alpha$, while the last term in~\eqref{eq:thm_xk2} is associated with noisy communication and decreases with $\alpha$. 
The size of the error neighborhood increases with the condition number $\kappa$ as expected, while the dependence on the algorithm parameter $t$ indicates that performing additional consensus steps mitigates the error due to network connectivity and the errors induced by the operators $\mathcal{T}_g\left[\cdot\right]$ and $\mathcal{T}_c\left[\cdot\right]$.

In the next corollary, we will use Theorem~\ref{thm:bounded_dist_min} and Lemmas~\ref{lem:descent_avg_iter} and~\ref{lem:bounded_iterates} to show that the the local iterates produced by S-NEAR-DGD$^t$ also have bounded distance to the solution of Problem~(\ref{eq:prob_orig}). 

\begin{cor}
\textbf{(Convergence of local iterates (S-NEAR-DGD$^t$)}
\label{cor:local_dist}
Let $x^t_{i,k}$ and $y_{i,k}$ be the local iterates generated by S-NEAR-DGD$^t$ at node $i$ and iteration $k$ from initial point $\*x_0=\*y_0=[y_{1,0};...;y_{n,0}] \in \mathbb{R}^{np}$ and let the steplength $\alpha$ satisfy 
\begin{equation*}
    \alpha < \min \left \{\frac{2}{\mu+L},\frac{2}{\mu_{\bar{f}}+L_{\bar{f}}} \right\},
\end{equation*}
where $\mu=\min_i\mu_i$, $L=\max_i L_i$, $\mu_{\bar{f}}=\frac{1}{n}\sum_{i=1}^n\mu_i$ and $L_{\bar{f}}=\frac{1}{n}\sum_{i=1}^n L_i$.

Then for $i=1,...,n$ and $k\geq1$ the distance of the local iterates to the solution of Problem~(\ref{eq:consensus_prob}) is bounded, i.e.
\begin{equation*}
    \begin{split}
        \mathbb{E}\Big[\big \|x^t_{i,k} - x^\star \big\|^2 \Big]
        &\leq 2\rho^k\mathbb{E} \left[\left\| \bar{x}_0 - x^\star\right\|^2\right] + 2\beta^{2t}\left(1+\frac{C}{n}\right)D + \frac{2\alpha\sigma_g^2}{n\gamma_{\bar{f}}}\\
        &\quad+ \frac{\beta^{2t}\left(1+\kappa\right)^2\left(n+C\right)\sigma_g^2}{L^2}  + \frac{8(n+C)\sigma_c^2}{1-\beta^2}+
        \frac{4\beta^{2t}(1+\kappa)^2(n+C)\sigma_c^2}{\alpha^2(1-\beta^2)L^2}, 
      \end{split}
\end{equation*}
and,
\begin{equation*}
    \begin{split}
        \mathbb{E}\Big[\big \|y_{i,k} - x^\star \big\|^2 \Big]
        &\leq  2\rho^k \mathbb{E}\left[ \left\|  \bar{x}_0 - x^\star\right\|^2\right] +2\left(1  + \frac{\beta^{2t}C}{n} \right)D  + \frac{2\alpha \sigma_g^2 }{n\gamma_{\bar{f}}} \\
        &\quad+   \frac{(1+\kappa)^2\left(n  + \beta^{2t}C \right)\sigma_g^2 }{L^2} + \frac{8 C \sigma_c^2}{1-\beta^2}  +  \frac{ 4(1+\kappa)^2\left(n  + \beta^{2t}C \right)\sigma_c^2}{\alpha^2(1-\beta^2) L^2},
        \end{split}
\end{equation*}
where $C =\frac{ \rho L^2}{\gamma_{\bar{f}}^2}$, $\rho=1-\alpha\gamma_{\bar{f}}$, $\gamma_{\bar{f}}=\frac{\mu_{\bar{f}}L_{\bar{f}}}{\mu_{\bar{f}}+L_{\bar{f}}}$, $\bar{x}_0=\sum_{i=1}^n y_{i,0}$ and the constant $D$ is defined in Lemma~\ref{lem:bounded_iterates}.
\end{cor}

\begin{proof}
For all $i \in \{1,...n\}$ and $k\geq1$, the following relation holds for the $x^t_{i,k}$ iterates,
\begin{equation}
\label{eq:lem_dist_min_local1}
    \begin{split}
        \left \|x^t_{i,k} - x^\star \right\|^2 &=  \left \|x^t_{i,k} -\bar{x}^t_k + \bar{x}^t_k - x^\star \right\|^2 \\
        & \leq  2\left \|x^t_{i,k} -\bar{x}^t_k\right\|^2 +2\left\| \bar{x}^t_k - x^\star \right\|^2,
    \end{split}
\end{equation}
where we added and subtracted $\bar{x}^t_k$ to get the first equality.

Taking the total expectation on both sides of (\ref{eq:lem_dist_min_local1}) and applying Lemma~\ref{lem:bounded_variance}, Theorem~\ref{thm:bounded_dist_min} and the definition of $C$ yields the first result of this corollary.

Similarly, for the $y_{i,k}$ local iterates we have,
\begin{equation}
\label{eq:lem_dist_min_local3}
    \begin{split}
        \left \|y_{i,k} - x^\star \right\|^2 &=  \left \|y_{i,k}-\bar{y}_{k} + \bar{y}_{k} - x^\star \right\|^2 \\
        &\leq 2 \left\|y_{i,k}-\bar{y}_{k}\right\|^2 +2\left\| \bar{y}_{k} - x^\star \right\|^2\\
        & =  2 \left\|y_{i,k}-\bar{y}_{k}\right\|^2 +2\left\| \bar{x}^t_{k} - x^\star \right\|^2,
        \end{split}
\end{equation}
where we derive the first equality by adding and subtracting $\bar{y}_{k}$ and used~\eqref{eq:near_dgd_x_avg} to obtain the last equality. 

Taking the total expectation on both sides of (\ref{eq:lem_dist_min_local3}) and applying Theorem~\ref{thm:bounded_dist_min}, Lemma~\ref{lem:bounded_variance} and the definition of $C$ completes the proof.
\end{proof}

Corollary~\ref{cor:local_dist} concludes our analysis of the S-NEAR-DGD$^t$ method. For the remainder of this section, we derive the convergence properties of S-NEAR-DGD$^+$, i.e. $t(k)=k$ for $k\geq1$ in~\eqref{eq:near_dgd_y} and~\eqref{eq:near_dgd_x}. 

\begin{thm}\textbf{(Convergence of S-NEAR-DGD$^+$)}
\label{thm:near_dgd_plus}
Consider the S-NEAR-DGD$^+$ method, i.e. $t(k)=k$ for $k\geq 1$. Let $\bar{x}^k_k=\frac{1}{n}\sum_{i=1}^n x^k_{i,k}$ be the average iterates produced by S-NEAR-DGD$^+$ and let the steplength $\alpha$ satisfy
\begin{equation*}
    \alpha < \min \left \{\frac{2}{\mu+L},\frac{2}{\mu_{\bar{f}}+L_{\bar{f}}} \right\}.
\end{equation*}
Then the distance of $\bar{x}^k_k$ to $x^\star$ is bounded for $k=1,2,...$, namely
\begin{equation*}
    \begin{split}
        &\mathbb{E}\Big[\big\|\bar{x}^k_k - x^\star \big\|^2\Big] \leq \rho^k \mathbb{E} \left[ \left\|\bar{x}_0 - x^\star \right\|^2 \right]+ \frac{\eta \theta^k\alpha\rho L^2D}{n\gamma_{\bar{f}}} + \frac{\alpha\sigma_g^2}{n\gamma_{\bar{f}}}+ \frac{\eta\theta^k\alpha\left(1+\kappa\right)^{2} \rho \sigma_g^2}{2\gamma_{\bar{f}}}  + \frac{4\rho L^2 \sigma_c^2}{\left(1-\beta^2\right)\gamma_{\bar{f}}^2}   + \frac{2\eta\theta^k\left(1+\kappa\right)^{2} \rho \sigma_c^2}{\alpha\left(1-\beta^2\right)\gamma_{\bar{f}}},
    \end{split}
\end{equation*} 
where $ \eta = \left|\beta^2-\rho\right|^{-1}$ and $\theta = \max\left\{\rho,\beta^2\right\}$.

\end{thm}

\begin{proof}
Replacing $t$ with $k$ in (\ref{eq:thm_xk1}) in Theorem~\ref{thm:bounded_dist_min} yields,
\begin{equation*}
    \begin{split}
     \mathbb{E}\Big[&\big\|\bar{x}^{k+1}_{k+1} - x^\star\big\|^2  \Big]
    \leq  \rho \mathbb{E}\left[\left\| \bar{x}^k_{k} - x^\star\right\|^2\right] + \frac{\alpha\beta^{2k}\rho L^2D}{n\gamma_{\bar{f}}}+\frac{\alpha^2\sigma_g^2}{n}+ \frac{\alpha\beta^{2k}\left(1+\kappa\right)^{2}\rho\sigma_g^2}{2\gamma_{\bar{f}}} + \frac{4\alpha\rho L^2\sigma_c^2}{\left(1-\beta^2\right)\gamma_{\bar{f}}} + \frac{2\beta^{2k}\left(1+\kappa\right)^{2}\rho \sigma_c^2}{\alpha\left(1-\beta^2\right)\gamma_{\bar{f}}}.
    \end{split}
\end{equation*}
Applying recursively for iterations $1, 2,\ldots, k$, we obtain,
\begin{equation}
\label{eq:thm_near_dgd_p_1}
    \begin{split}
        \mathbb{E}\Big[&\big\|\bar{x}^k_k - x^\star \big\|^2\Big] \leq \rho^k \mathbb{E} \left[ \left\|\bar{x}_0 - x^\star \right\|^2 \right]  + S_1 \left(\frac{\alpha\rho L^2D}{n\gamma_{\bar{f}}} +   \frac{\alpha\left(1+\kappa\right)^2\rho \sigma_g^2}{2\gamma_{\bar{f}}} +\frac{2\left(1+\kappa\right)^2\rho \sigma_c^2}{\alpha\left(1-\beta^2\right)\gamma_{\bar{f}}}\right) +  S_2 \left(\frac{\alpha^2\sigma_g^2}{n} + \frac{4\alpha\rho L^2\sigma_c^2}{\left(1-\beta^2\right)\gamma_{\bar{f}}} \right)\\
    \end{split}
\end{equation}
where $S_1 = \sum_{j=0}^{k-1} \rho^j \beta^{2(k-1-j)}$ and $S_2 = \sum_{j=0}^{k-1}\rho^j$.

Let $\psi = \frac{\rho}{\beta^2}$. Then $S_1 = \beta^{2(k-1)}\sum_{j=0}^{k-1} \psi^j= \beta^{2(k-1)}\frac{1-\psi^k}{1-\psi}= \frac{\beta^{2k}-\rho^k}{\beta^2 - \rho}\leq \eta \theta^k$.
Applying this result and the bound $S_2 \leq \frac{1}{1-\rho} = \frac{1}{\alpha \gamma_{\bar{f}}}$ to (\ref{eq:thm_near_dgd_p_1}) yields the final result.
\end{proof}

Theorem~\ref{thm:near_dgd_plus} indicates that S-NEAR-DGD$^+$ converges with geometric rate $\theta=\max \left\{ \rho , \beta^2 \right \}$ to a neighborhood of the optimal solution $x^\star$ of Problem~\eqref{eq:prob_orig} with size
\begin{equation}
\label{eq:plus_err_size}
    \lim_{k\rightarrow \infty} \sup \mathbb{E}\left[\left\|\bar{x}^k_k - x^\star \right\|^2\right] = \frac{\alpha\sigma_g^2}{n\gamma_{\bar{f}}}  + \frac{4\rho L^2 \sigma_c^2}{\left(1-\beta^2\right)\gamma_{\bar{f}}^2}.
\end{equation}
The first error term on right-hand side  of Eq.~\eqref{eq:plus_err_size} depends on the variance of the gradient error $\sigma_g$ and is inversely proportional to the network size $n$. This scaling with $n$, which has a similar effect to centralized mini-batching, is a trait that our method shares with a number of distributed stochastic gradient algorithms. The last error term depends on the variance of the communication error $\sigma_c$ and increases with $\beta$, implying that badly connected networks accumulate more communication error over time.

Conversely, Eq.~\eqref{eq:thm_xk2} of Theorem~\ref{thm:bounded_dist_min} yields,
\begin{equation}
\label{eq:t_err_size}
\begin{split}
     \lim_{k\rightarrow \infty} &\sup \mathbb{E}\left[\left\|\bar{x}^t_k - x^\star \right\|^2\right] = \frac{\alpha \sigma_g^2  }{n\gamma_{\bar{f}}} +  \frac{4\rho L^2\sigma_c^2}{\left(1-\beta^2\right)\gamma_{\bar{f}}^2}+  \frac{\beta^{2t}\rho}{\gamma_{\bar{f}}^2}\left(\frac{L^2 D}{n}+ \frac{(1+\kappa)^2 \sigma_g^2}{2}+\frac{2(1+\kappa)^2\sigma_c^2}{\alpha(1-\beta^2)}\right).
    \end{split}
\end{equation}
Comparing~\eqref{eq:plus_err_size} and~\eqref{eq:t_err_size}, we observe that~\eqref{eq:t_err_size} contains three additional error terms, all of which depend directly on the algorithm parameter $t$. Our results imply that S-NEAR-DGD$^t$ generally converges to a worse error neighborhood than S-NEAR-DGD$^+$, and approaches the error neighborhood of S-NEAR-DGD$^+$ as $t\rightarrow \infty$.

\section{Numerical results}
\label{sec:numerical}

\subsection{Comparison to existing algorithms}
To quantify the empirical performance of S-NEAR-DGD, we consider the following regularized logistic regression problem to classify the mushrooms dataset \cite{mushrooms},
\begin{equation*}
 \min_{x \in \mathbb{R}^p} f(x) = \frac{1}{M}\sum_{s=1}^M \log(1+e^{-b_s \langle A_s , x \rangle}) + \frac{1}{M}\|x\|^2_2,
\end{equation*}
where $M$ is the total number of samples ($M=8124$), $A \in \mathbb{R}^{M \times p}$ is a feature matrix, $p$ is the problem dimension ($p=118$) and $b \in \{-1,1\}^{M}$ is a vector of labels. 

To solve this problem in a decentralized manner, we evenly distribute the samples among $n$ nodes and assign to node $i \in \{1,...,n\}$ the private function $f_i(x) = \frac{1}{|S_i|}\sum_{s \in S_i} \log(1+e^{-b_s \langle A_s , x \rangle}) + \frac{1}{M}\|x\|^2_2$,
where $S_i$ is the set of sample indices accessible to node $i$. 
For the network topology we selected a connected, random network of $n=14$ nodes generated with the Erd\H{o}s-R\'{e}nyi model (edge probability $0.5$). To construct the stochastic gradient approximations $g_{i,k-1}=\mathcal{T}_g\left[\nabla f_i \left(x^{t(k-1)}_{i,k-1}\right)\right]$, each node randomly samples with replacement a batch of size $B=16$  
from its local distribution and computes a mini-batch gradient. To simulate the inexact communication operator $\mathcal{T}_c \left[\cdot\right]$ we implemented the probabilistic quantizer of Example~\ref{ex:prob_q}.

Moreover, to test the effects of different approaches to handle the communication noise/quantization error, we implemented three schemes as listed in Table~\ref{tab:consensus_variants}. Specifically, variant~Q.1 is what S-NEAR-DGD uses in step \eqref{eq:near_dgd_x}, and includes a consensus step using the quantized noisy variables $q_{i,k}^j$ and then adding the error correction term $\left(x^{j-1}_{i,k}-q^j_{i,k}\right)$. Variant~Q.2 considers a more na\"{\i}ve approach, where a simple weighted average of the noisy variables $q_{i,k}$ is calculated without the addition of error correction. Finally, in variant~Q.3 we assume that node $i$ either does not have access to its local quantized variable $q_{i,k}$ or prefers to use the original quantity $x^{j-1}_{i,k}$ whenever possible and thus computes the weighted average using its local unquantized variable and quantized versions from its neighbors. For algorithms that perform consensus step on gradients instead of the decision variables, similar schemes were implemented. 

\begin{table}[ht]
\begin{center}
\begin{tabular}{ c c  }
\hline
 \textbf{Variant name} & \textbf{Consensus update}\\ 
 \hline
Q.1 &  $x^j_{i,k} \leftarrow \sum_{l=1}^n \left(w_{il}  q^{j}_{l,k} \right) + \left( x^{j-1}_{i,k} - q^j_{i,k}\right)$ \\ Q.2 & $x^j_{i,k} \leftarrow \sum_{l=1}^n \left(w_{il}  q_{i,k}^j\right)$  \\
Q.3 & $x^j_{i,k} \leftarrow w_{ii}x^{j-1}_{i,k} + \sum_{l \in \mathcal{N}_i}\left(w_{il}  q^j_{l,k}\right)$   \\
 \hline
\end{tabular}
\caption{Quantized consensus step variations}
\label{tab:consensus_variants}
\end{center}
\end{table}
We compared the S-NEAR-DGD$^t$ method with $t=2$ and $t=5$, to versions of DGD \cite{NedicSubgradientConsensus,sundharram_distributed_2010}, EXTRA~\cite{extra} and DIGing~\cite{diging} with noisy gradients and communication. All methods use the same mini-batch gradients at every iteration and exchange variables that are quantized with the same stochastic protocol. We implemented the consensus step variations Q.1, Q.2 and Q.3 for all methods (we note that combining DIGing with Q.1 and using the true local gradients
recovers the iterates of Q-NEXT~\cite{lee2018finite}. However, the authors of~\cite{lee2018finite} accompany their method with a dynamic quantizer that we did not implement for our numerical experiments
). All algorithms shared the same steplength ($\alpha=1$)\footnote{Smaller steplengths yield data trends like those shown in Fig.~\ref{fig:all_methods}.}.  

In Figure~\ref{fig:all_methods}, we plot the squared error $\|\bar{x}_k - x^\star\|^2$ versus the number of algorithm iterations (or gradient evaluations) for quantization parameter values $\Delta = 10$~(right) and $\Delta = 10^5$~(left). The most important takeaway from Fig.~\ref{fig:all_methods} is that the two gradient tracking (GT) methods, namely EXTRA and DIGing, diverge without error correction (variants Q.2 and Q.3). Conversely, purely primal methods such as DGD and S-NEAR-DGD appear to be more robust to quantization error when error correction is not incorporated. This observation aligns with recent findings indicating that GT methods are more sensitive to network topology (the value of $\beta$, on which the quantization error strongly depends) compared to primal methods~\cite{Yuan2020CanPM}. Our intuition suggests that variant Q.1 achieves better performance than Q.2 and Q.3 by cancelling out the average quantization error and preventing it from accumulating over time (for an analysis of S-NEAR-DGD under variant Q.2, we refer readers to Section~\ref{appendix:q2} in the Appendix). We notice negligible differences between variants Q.2 and Q.3 for the same algorithm. All methods achieved comparable accuracy when combined with Q.1 and the quantization was relatively fine (Fig. \ref{fig:all_methods}, right), with the exception of DGD which converges further away from the optimal solution. For relatively coarse quantization (Fig.~\ref{fig:all_methods}, left), the two variants of S-NEAR-DGD and EXTRA perform marginally better than DGD and DIGing when combined with variant~Q.1.

\begin{figure}
    \centering
    \includegraphics[width=0.8\textwidth]{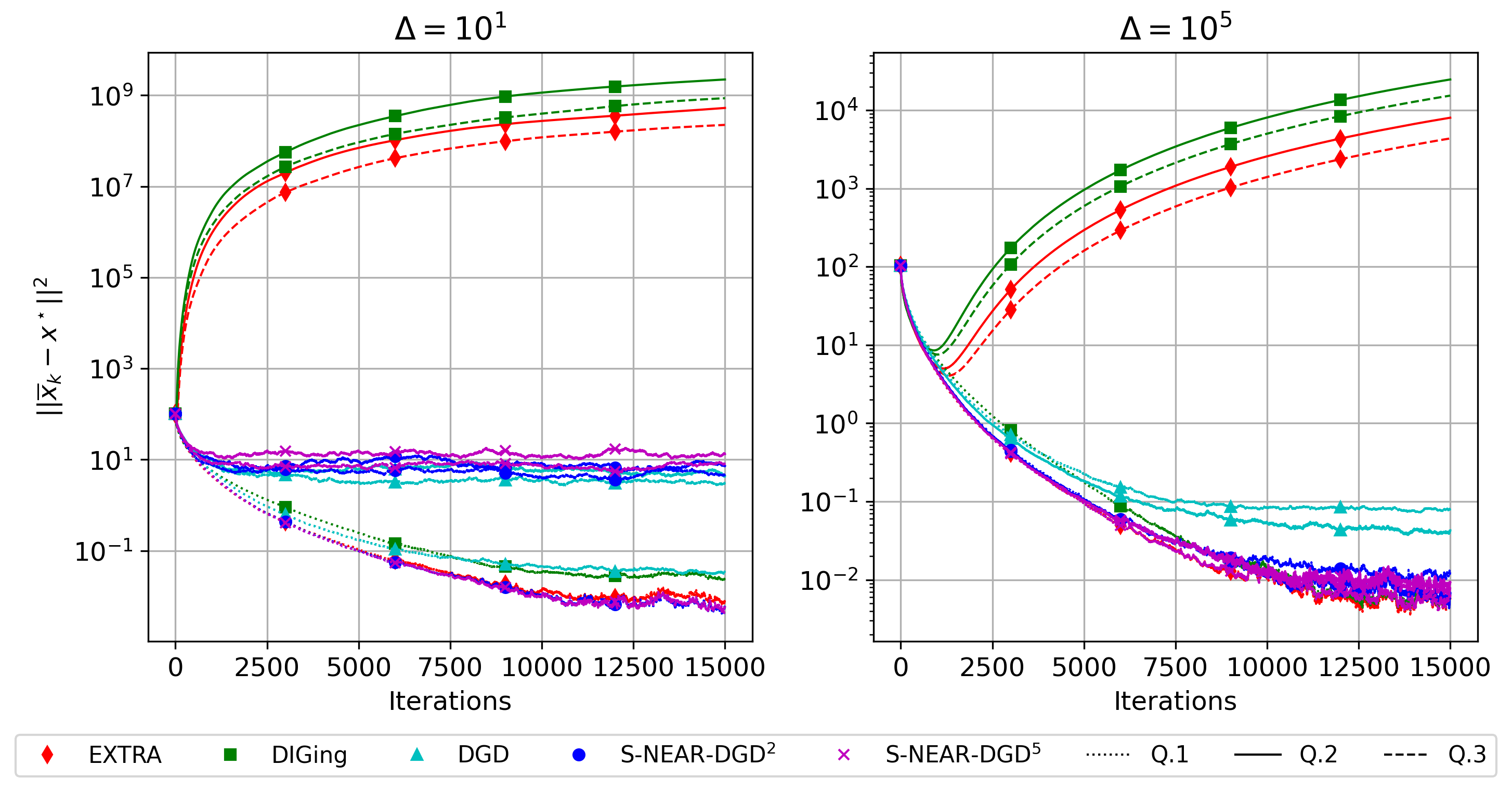}
    \caption{Error plots, $\Delta=10$ (left) and $\Delta=10^5$ (right)}
    \label{fig:all_methods}
\end{figure}

\subsection{Scalability}


To evaluate the scalability of our method and the effect of network type on convergence accuracy and speed, we tested $5$ network sizes ($n=5,10,15,20,25$) and $5$ different network types: $i$) complete, $ii$) random (connected Erd\H{o}s-R\'{e}nyi, edge probability $0.4$), $iii$) 4-cyclic (i.e. starting from a ring graph, every node is connected to $4$ immediate neighbors), $iv$) ring and $v$) path. We compared $2$ instances of the NEAR-DGD$^t$ method, $i$) $t=1$ and $ii$) $t=7$. We opted to exclude NEAR-DGD$^+$ from this set of experiments to facilitate result interpretability in terms of communication load. We set $\alpha = 1$ and $\Delta=10^2$ for all instances of the experiment, while the batch size for the calculation of the local stochastic gradients was set to $B=16$ in all cases. Different methods applied to networks of identical size selected (randomly, with replacement) the same samples at each iteration.

Our results are summarized in Figure~\ref{fig:scaling}. In Fig.~\ref{fig:scaling}, top left, we terminated all experiments after  $T=2\cdot 10^4$ iterations and plotted the normalized function value error $\left(f\left(\bar{x}_k\right) - f\left(x^\star\right)\right)/f(x^\star)$, averaged over the last $\tau=10^3$ iterations. We observe that networks with better connectivity converge closer to the true optimal value, implying that the terms inversely proportional to $(1-\beta^2)$ dominate the error neighborhood in Eq.~\eqref{eq:t_err_size}. Adding more nodes improves convergence accuracy for well-connected graphs (complete, random), possibly due to the "variance reduction" effect on the stochastic gradients discussed in the previous section. For the remaining graphs, however, this beneficial effect is outweighed by the decrease in connectivity that results from the introduction of additional nodes and consequently, large values of $n$ yield worse convergence neighborhoods.
Increasing the number of consensus steps per iteration has a favorable effect on accuracy, an observation consistent with our theoretical findings in the previous section.

For the next set of plots, we analyze the run time and cost of the algorithm until termination. The presence of stochastic noise makes the establishment of a termination criterion that performs well for all parameter combinations a challenging task. Inspired by~\cite{welford_method}, we tracked an approximate time average $\bar{f}$ of the function values $f\left(\bar{x}_k\right)$ using Welford's method for online sample mean computation, i.e. $\bar{f}_k = \bar{f}_{k-1} + \frac{f(\bar{x}_k) - \bar{f}_{k-1}}{k}$, for $k=1,2,...$, with $\bar{f}_0 = f\left(\bar{x}_0\right)$. We terminate the algorithm at iteration count $k$ if $ \left| \frac{\bar{f}_k - \bar{f}_{k-1}}{\bar{f}_{k-1}} \right| < \epsilon$, where $\epsilon$ is a tolerance parameter. 

In Figure~\ref{fig:scaling}, top right, we graph the number of steps (or gradient evaluations) until the termination criterion described in the previous paragraph is satisfied for $\epsilon=10^{-5}$. We observe a similar trend to Figure~\ref{fig:scaling}, top left, indicating that poor connectivity has an adverse effect on both accuracy and the rate of convergence, although the latter is not predicted by the theory.
Increasing the number of consensus steps per iteration reduces the total number of steps needed to satisfy the stopping criterion. Finally, in the bottom row of Fig.~\ref{fig:scaling} we plot the total application cost per node until termination, which we calculated using the cost framework first introduced in~\cite{berahas_balancing_2019},
\begin{equation*}
\label{eq:cost_frame}
    \text{Cost} = c_c \times \# \text{Communications} +  c_g \times \# \text{Computations},
\end{equation*}
where $c_c$ and $c_g$ are constants representing the application-specific costs of communication and computation respectively. 

In Fig.~\ref{fig:scaling}, bottom right, the communication is a $100$ times cheaper than computation, i.e. $c_c=0.01 \cdot c_g$. Increasing the number of consensus steps per iteration almost always yields faster convergence in terms of total cost, excluding some cases where the network is already reasonably well-connected (eg. complete graph). In Fig.~\ref{fig:scaling}, bottom left, the costs of computation and communication are equal, i.e. $c_c=c_g$, and increasing the number of consensus steps per iteration results in higher total cost in all cases.



\begin{figure}
    \centering
    \includegraphics[width=0.8\textwidth]{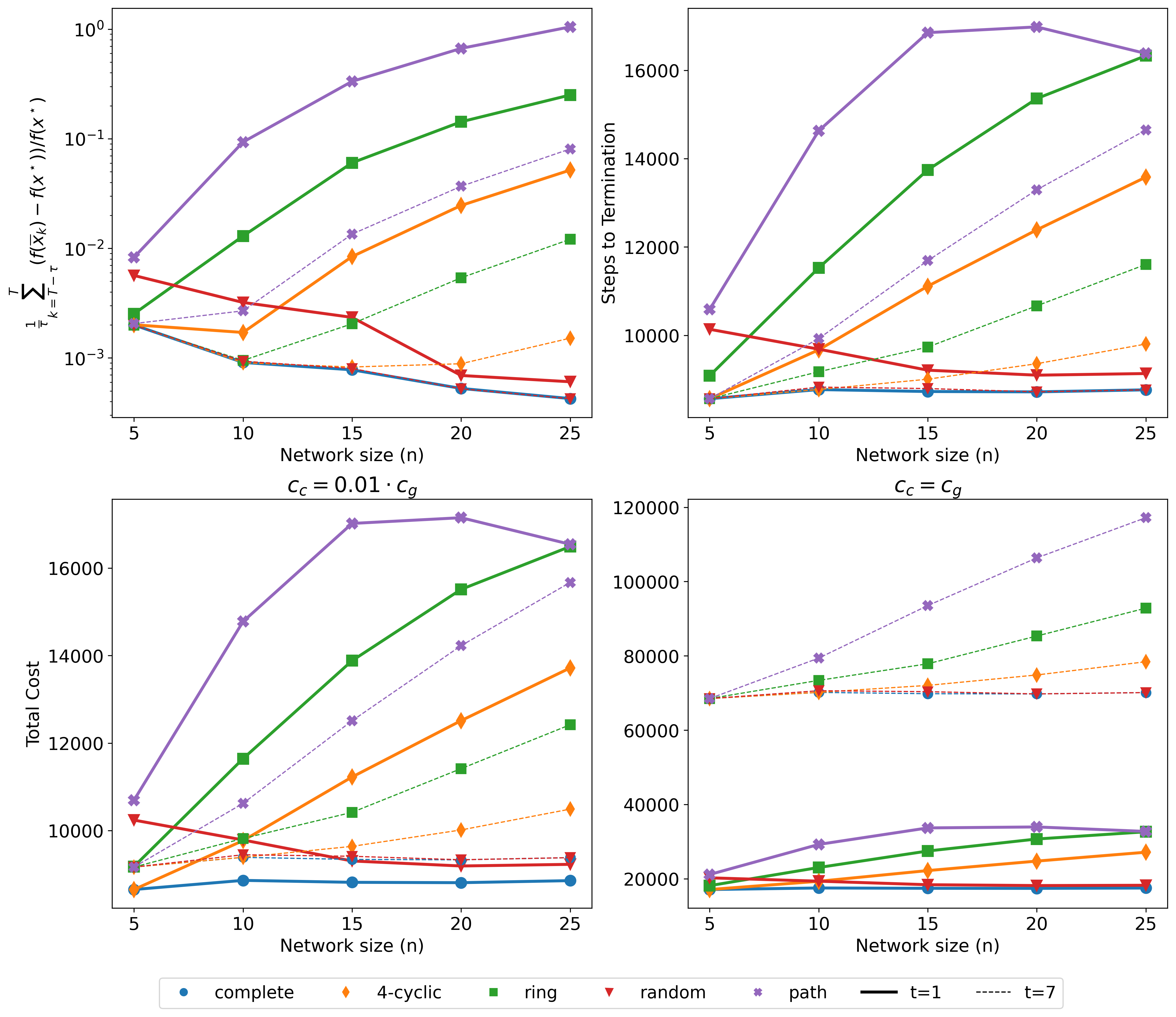}

    \caption{Dependence on network type and size. Function value error averaged over the last $\tau$ iterations out of $T$ total iterations (top left), steps/gradient evaluations until termination (top right), total cost until termination when communication is cheaper than computation (bottom left), and when communication and computation have the same cost (bottom right).}
    \label{fig:scaling}
\end{figure}

\section{Conclusion}
\label{sec:conclusion}

We have proposed a first order method (S-NEAR-DGD) for distributed optimization over fixed, undirected networks, that can tolerate stochastic gradient approximations and noisy information exchange to alleviate the loads of computation and communication, respectively. The strength of our method lies in its flexible framework, which alternates between gradient steps and a number of nested consensus rounds that can be adjusted to match application requirements. Our analysis indicates that S-NEAR-DGD converges in expectation to a neighborhood of the optimal solution when the local objective functions are strongly convex and have Lipschitz continuous gradients. We have quantified the dependence of this neighborhood on algorithm parameters, problem-related quantities and the topology and size of the network. Empirical results demonstrate that our algorithm performs comparably or better than state-of-the-art methods, depending on the implementation of stochastic quantized consensus. Future directions include the application of this method to nonconvex settings and directed graphs.

\appendix

In this section, we analyze S-NEAR-DGD under the quantized consensus variant Q.2 of Table~\ref{tab:consensus_variants}. We retain the same notation and assumptions with the main part of this work. 

The system-wide iterates of S-NEAR-DGD when consensus is implemented with variant Q.2 can be expressed as,
\begin{subequations}
\begin{align}
    &\*y_{k} = \*x_{k-1}^{t(k-1)} - \alpha \*g_{k-1}\label{eq:near_dgd_y_q2},\\
    &\*x^j_k =  \*Z \*q^{j}_k,\quad j=1,...,t(k), \label{eq:near_dgd_x_q2}
\end{align}
\end{subequations}
where $\*g_{k-1}=\left[g_{1,k-1};...;g_{n,k-1}\right]$ , $g_{i,k-1}=\mathcal{T}_g\left[\nabla f_i \left(x^{t(k-1)}_{i,k-1}\right)\right]$, $\*Z = \left(\*W \otimes I_p\right) \in \mathbb{R}^{np \times np}$, $\*q^j_k=\left[q^j_{1,k};...;q^j_{n,k}\right]$ for $j=1,...,t(k)$ and $q^{j}_{i,k}=\mathcal{T}_c\left[x^{j-1}_{i,k}\right]$. 

In the next subsection, we derive the convergence properties of the iterates generated by~\eqref{eq:near_dgd_y_q2} and~\eqref{eq:near_dgd_x_q2}. We closely follow the proof structure of our main analysis, excluding those Lemmas that are independent of the implementation of consensus (Lemmas~\ref{lem:descent_f}, \ref{lem:global_L_mu}, \ref{lem:comp_error} and~\ref{lem:descent_avg_iter}). In Lemma~\ref{lem:comm_error_q2}, we show that the total communication error in a single iteration of S-NEAR-DGD$^t$ ($t(k)=t$ in~\eqref{eq:near_dgd_y_q2} and~\eqref{eq:near_dgd_x_q2}) is proportional to the parameter $t$, the number of consensus rounds performed. We derive upper bounds for the magnitude of the iterates and the distance of the local iterates to the average iterates in Lemmas~\ref{lem:bounded_iterates_q2} and~\ref{lem:bounded_variance_q2}, respectively. In Theorem~\ref{thm:bounded_dist_min_q2}, we prove that the average iterates $\bar{x}^t_k =\frac{1}{n}\sum_{i=1}^n x^t_{i,k}$ of S-NEAR-DGD$^t$ under Q.2 converge (in expectation) with geometric rate to a neighborhood of the optimal solution. Specifically, this convergence neighborhood can be described by,
\begin{equation}
\label{eq:t_err_size_q2}
\begin{split}
     \lim_{k\rightarrow \infty} \sup \mathbb{E}\left[\left\|\bar{x}^t_k - x^\star \right\|^2\right] &= \frac{\beta^{2t}\rho L^2D}{n\gamma_{\bar{f}}^2} + \frac{\alpha \sigma_g^2  }{n\gamma_{\bar{f}}}+  \frac{\beta^{2t} \left(1+\kappa \right)^2\rho\sigma_g^2}{2\gamma_{\bar{f}}^2} + \frac{t\sigma_c^2}{n\alpha\gamma_{\bar{f}}} + \frac{\rho L^2t\sigma_c^2}{\gamma_{\bar{f}}^2} +  \frac{\beta^{2t}\left(1+\kappa\right)^2\rho t\sigma_c^2}{2\alpha^2\gamma_{\bar{f}}^2},
    \end{split}
\end{equation}
We observe that several terms in~\eqref{eq:t_err_size_q2} are proportional to $t$, the number of consensus rounds per iteration. Conversely, when an error correction mechanism is incorporated to consensus,~\eqref{eq:t_err_size} suggests that error terms are either independent of or decrease with $t$. 

Next, we prove that the local iterates of S-NEAR-DGD$^t$ under Q.2 converge to a neighborhood of the optimal solution in Corollary~\ref{cor:local_dist_q2}. Finally, we establish an upper bound for the distance of the average iterates $\bar{x}^k_k =\frac{1}{n}\sum_{i=1}^n x^k_{i,k}$ of S-NEAR-DGD$^+$ ($t(k)=k$ in~\eqref{eq:near_dgd_y_q2} and~\eqref{eq:near_dgd_x_q2}) to the optimal solution in Theorem~\ref{thm:near_dgd_plus_q2}. Our theoretical results indicate that this method diverges, as the distance to the solution increases at every iteration.

\subsection{Convergence analysis of S-NEAR-DGD under consensus variant Q.2}
\label{appendix:q2}

\begin{lem}\textbf{(Bounded communication error (Q.2))}
\label{lem:comm_error_q2}
Let $\mathcal{E}^c_{t,k} := \*x^{t}_k - \*Z^{t}\*y_k$ be the total communication error at the $k$-th iteration of S-NEAR-DGD$^t$ under quantization variant Q.2, i.e. $t(k)=t$ in~\eqref{eq:near_dgd_y_q2} and~\eqref{eq:near_dgd_x_q2}. Then the following relations hold for $k=1,2,...$
\begin{equation*}
    \begin{split}
      \mathbb{E}_{\mathcal{T}_c}\left[\mathcal{E}^{c}_{t,k} \big| \mathcal{F}^0_k \right] = \mathbf{0}, \quad \mathbb{E}_{\mathcal{T}_c}\left[\left \|\mathcal{E}^{c}_{t,k}  \right\|^2  \big| \mathcal{F}^0_k \right]
        &\leq nt\sigma_c^2.
    \end{split}
\end{equation*}
\end{lem}

\begin{proof} Setting $\*x^0_k=\*y_k$, we can re-write $\mathcal{E}^{c}_{t,k}$ as,
\begin{equation}
\label{eq:lem_comp_err_1_q2}
    \begin{split}
    \mathcal{E}^{c}_{t,k} &= \*Z \*q^{t}_k - \*Z^t \*x^0_k\\
    &= \*Z \*q^{t}_k - \sum_{j=1}^{t-1}\*Z^{t-j}\*x_k^{j} + \sum_{j=1}^{t-1}\*Z^{t-j}\*x_k^{j} -  \*Z^t \*x^0_k\\\
    &= \*Z \*q^{t}_k - \sum_{j=1}^{t-1}\*Z^{t-j}\*x_k^{j} + \sum_{j=1}^{t-1}\*Z^{t-j+1}\*q_k^{j} -  \*Z^t \*x^0_k\\\
    &=\sum_{j=1}^{t} \*Z^{t-j+1}\left(\*q^{j}_k-\*x^{j-1}_k \right),
    \end{split}
\end{equation}
where we used (\ref{eq:near_dgd_x_q2}) to get the first equality, added and subtracted the quantity $\sum_{j=1}^{t-1}\*Z^{t-j}\*x_k^{j}$ to get the second equality and used (\ref{eq:near_dgd_x_q2}) again for the third equality. 
We notice that $\*q^j_k-\*x^{j-1}_k=\left[\epsilon^j_{i,k};...;\epsilon^j_{n,k}\right]$ as defined in Assumption~\ref{assum:tc_bound}. For $j= 1,\hdots, t(k)$ we therefore have, 
\begin{equation*}
     \mathbb{E}_{\mathcal{T}_c}\left[ \*q^j_k-\*x^{j-1}_k  \Big| \mathcal{F}^{j-1}_k\right] = \mathbf{0}.
\end{equation*}
Due to the fact that $\mathcal{F}^{0}_k \subseteq \mathcal{F}^{1}_k \subseteq ... \subseteq \mathcal{F}^{j-1}_k$, applying the tower property of conditional expectation we yields,
\begin{equation}
\label{eq:comm_err_1st_mom_q2}
    \begin{split}
         \mathbb{E}_{\mathcal{T}_c}\left[ \*q^j_k-\*x^{j-1}_k  \Big| \mathcal{F}^0_k\right] = \mathbb{E}_{\mathcal{T}_c}\left[ \mathbb{E}_{\mathcal{T}_c}\left[ \*q^j_k-\*x^{j-1}_k \Big| \mathcal{F}^{j-1}_k\right] \Big| \mathcal{F}^0_k\right] = \mathbf{0}, \quad j=1,...,t,
    \end{split}
\end{equation}
and hence  $\mathbb{E}_{\mathcal{T}_c}\left[ \mathcal{E}^c_{t,k}  \big| \mathcal{F}^0_k\right] = \mathbf{0}$ due to linearity of expectation and Eq. \eqref{eq:lem_comp_err_1_q2}.
Moreover, using the fact that $\left\|\*Z\right\|=1$ due to Assumption~\ref{assum:consensus_mat}, 
we obtain for $j=1,...,t$,
\begin{equation}
\label{eq:comm_err_2nd_mom_q2}
    \begin{split}
         \mathbb{E}_{\mathcal{T}_c}\left[ \left\| \*Z^{t-j+1}\left(\*q^j_k-\*x^{j-1}_k \right) \right\|^2 \Big| \mathcal{F}^0_k\right] &\leq  \mathbb{E}_{\mathcal{T}_c}\left[ \left\| \*q^j_k-\*x^{j-1}_k  \right\|^2 \Big| \mathcal{F}^0_k\right]\\
         &=\mathbb{E}_{\mathcal{T}_c}\left[ \mathbb{E}_{\mathcal{T}_c}\left[ \left\| \*q^j_k-\*x^{j-1}_k \right\|^2 \Big| \mathcal{F}^{j-1}_k\right] \bigg| \mathcal{F}^0_k\right] \\
         &=\mathbb{E}_{\mathcal{T}_c}\left[  \sum_{i=1}^{n}  \mathbb{E}_{\mathcal{T}_c}\left[ \left\|\epsilon^j_{i,k} \right\|^2  \Big| \mathcal{F}^{j-1}_k\right] \bigg| \mathcal{F}^0_k\right]\\
         &\leq n\sigma_c^2,
    \end{split}
\end{equation}
where we derived the second inequality using the tower property of conditional expectation and applied Assumption~\ref{assum:tc_bound} to get the last inequality.

Assumption~\ref{assum:tc_bound} implies that for all $i=1,...,n$ and $j_1 \neq j_2$, $\epsilon^{j_1}_{i,k}$ and $\epsilon^{j_2}_{i,k}$ and by extension $\*q^{j_1}_k - \*x^{j_1-1}_k$ and $\*q^{j_2}_k - \*x^{j_2-1}_k$ are independent. Eq.~\eqref{eq:comm_err_1st_mom_q2} then yields,
\begin{equation*}
    \mathbb{E}_{\mathcal{T}_c}\left[\left \langle \*Z^{t-j_1+1}\left(\*q^{j_1}_k - \*x^{j_1-1}_k\right), \*Z^{t-j_1+1}\left(\*q^{j_1}_k - \*x^{j_1-1}_k\right) \right \rangle\bigg|\mathcal{F}^0_k\right] = \mathbf{0}.
\end{equation*}
Combining the equation above with~\eqref{eq:comm_err_2nd_mom_q2} and by linearity of expectation, after expanding the squared norm $\left\|\mathcal{E}^c_{t,k}\right\|^2$ we obtain,
\begin{equation*}
    \begin{split}
         \mathbb{E}_{\mathcal{T}_c} \left[\left\| \mathcal{E}^c_{t,k} \right\|^2 \big| \mathcal{F}^{0}_k \right] &=  \mathbb{E}_{\mathcal{T}_c} \left[ \left\|\sum_{j=1}^{t} \*Z^{t-j+1}\left(\*q^j_k-\*x^{j-1}_k \right)\right\|^2\Bigg| \mathcal{F}^{0}_k \right] \\
          &= \sum_{j=1}^{t} \mathbb{E}_{\mathcal{T}_c} \left[ \left\| \*Z^{t-j+1}\left(\*q^j_k-\*x^{j-1}_k \right)\right\|^2\Bigg| \mathcal{F}^{0}_k \right]\\
         &\leq nt\sigma_c^2,
    \end{split}
\end{equation*}
which completes the proof.
\end{proof}


We begin our convergence analysis by proving that the iterates generated by S-NEAR-DGD$^t$ are bounded in expectation.

\begin{lem}\textbf{(Bounded iterates (Q.2))}
\label{lem:bounded_iterates_q2} Let $\*y_k$ and $\*x_k$ be the iterates generated by (\ref{eq:near_dgd_y_q2}) and (\ref{eq:near_dgd_x_q2}), respectively, from initial point $\*y_0 \in \mathbb{R}^{np}$. Moreover, let the steplength $\alpha$ satisfy,
\begin{equation*}
    \alpha < \frac{2}{\mu + L},
\end{equation*}
where $\mu=\min_i \mu_i$ and $L=\max_i L_i$.

Then $\*x_k$ and $\*y_k$ are bounded in expectation for $k=1,2,...$, i.e.,
\begin{equation*}
    \begin{split}
    \mathbb{E}\left[\left\| \*y_{k}\right\|^2 \right]\leq D + 
      \frac{(1+\kappa)^2n\sigma_g^2}{2L^2}+ \frac{(1+\kappa)^2nt\sigma_c^2}{2\alpha^2 L^2},
    \end{split}
\end{equation*}
\begin{equation*}
    \begin{split}
        \mathbb{E}\left[\left\| \*x^t_k \right\|^2 \right] \leq D + 
      \frac{(1+\kappa)^2n\sigma_g^2}{2L^2}+ \frac{(1+\kappa)^2nt\sigma_c^2}{2\alpha^2 L^2}  + nt\sigma_c^2,
    \end{split}
\end{equation*}
where $D=2 \mathbb{E}\left[
      \left\|\*y_0 - \*u^\star \right\|^2 \right]+2\left(1+4\nu^{-3}\right)\left\|\*u^\star\right\|^2$, $\*u^\star=[u_1^\star;u_2^\star;...;u_n^\star] \in \mathbb{R}^{np}$,  $u_i^\star = \arg \min_x f_i(x)$, $\nu=\frac{2\alpha\mu L}{\mu+L}$ and $\kappa=L/\mu$ is the condition number of Problem~\eqref{eq:consensus_prob}.
\end{lem}

\begin{proof}
Consider,
\begin{equation*}
    \begin{split}
        \left\| \*y_{k+1} - \*u^\star \right\|^2 &=  \left\|\*x_k^{t} - \alpha \*g_k - \*u^\star \right\|^2\\
        &= \left\|\*x_k^{t} - \alpha \nabla \*f(\*x_k^{t}) - \*u^\star  - \alpha  \mathcal{E}^g_{k+1}  \right\|^2 \\
        &= \left\|\*x_k^{t} - \alpha \nabla \*f(\*x_k^{t}) - \*u^\star \right\|^2 + \alpha^2\left\| \mathcal{E}^g_{k+1} \right\|^2 - 2\alpha \left  \langle \*x_k^{t} - \alpha \nabla \*f(\*x_k^{t}) - \*u^\star, \mathcal{E}^g_{k+1}\right \rangle,
    \end{split}
\end{equation*}
where we used (\ref{eq:near_dgd_y_q2}) to get the first equality and added and subtracted $\alpha\nabla\*f\left(\*x^t_k\right)$ and applied the computation error definition $\mathcal{E}^g_{k+1}:=\*g_k - \nabla \*f\left(\*x^t_k\right)$ to obtain the second equality.

Taking the expectation conditional to $\mathcal{F}^t_k$ on both sides of
the relation above
and applying Lemma~\ref{lem:comp_error} yields,
\begin{equation}
\label{eq:lem_bounded_xy_2_q2}
    \begin{split}
        \mathbb{E}\left[\left\| \*y_{k+1} - \*u^\star \right\|^2 \Big| \mathcal{F}^t_k \right]&\leq \left\|\*x_k^{t} - \alpha \nabla \*f(\*x_k^{t}) - \*u^\star \right\|^2 + \alpha^2 n \sigma_g^2.
    \end{split}
\end{equation}
For the first term on the right-hand side of (\ref{eq:lem_bounded_xy_2_q2}), combining Lemma~\ref{lem:descent_f}
with Lemma~\ref{lem:global_L_mu} and due to $\alpha < \frac{2}{\mu+L}$, we have
\begin{equation*}
\label{eq:lem_bounded_xy_3}
  \left\|\*x_k^{t} - \alpha \nabla \*f(\*x_k^{t}) - \*u^\star  \right\|^2 \leq (1-\nu)\left\|\*x_k^{t}-\*u^\star\right\|^2,
\end{equation*}
where $\nu=\frac{2\alpha \mu L}{\mu + L}=\frac{2\alpha L}{1+\kappa}<1$.

Expanding the term on the right hand side of the previous relation yields,
\begin{equation*}
\label{eq:lem_bounded_xy_4}
    \begin{split}
      \left\|\*x_k^{t}-\*u^\star\right\|^2 &= \left\|\mathcal{E}^{c}_{t,k} + \*Z^t\*y_k -\*Z^t\*u^\star + \*Z^t\*u^\star - \*u^\star\right\|^2\\
      &= \left\|\mathcal{E}^{c}_{t,k}\right\|^2 + \left\|\*Z^t\left(\*y_k - \*u^\star\right) - \left(I-\*Z^t\right)\*u^\star\right\|^2 + 2\left \langle \mathcal{E}^{c}_{t,k}, \*Z^t\*y_k - \*u^\star  \right \rangle \\
      &\leq \left\|\mathcal{E}^{c}_{t,k}\right\|^2 + \left(1+\nu\right)\left\|\*Z^t\left(\*y_k - \*u^\star\right)\right\|^2 +\left(1+\nu^{-1}\right)\left\| \left(I-\*Z^t\right)\*u^\star\right\|^2+ 2\left \langle \mathcal{E}^{c}_{t,k}, \*Z^t\*y_k - \*u^\star \right \rangle\\
      &\leq \left\|\mathcal{E}^{c}_{t,k}\right\|^2 + \left(1+\nu\right)\left\|\*y_k - \*u^\star\right\|^2 +4\left(1+\nu^{-1}\right)\left\|\*u^\star\right\|^2 + 2\left \langle \mathcal{E}^{c}_{t,k}, \*Z^t\*y_k - \*u^\star  \right \rangle,
    \end{split}
\end{equation*}
where we added and subtracted the quantities $\*Z^t\*y_k$ and $\*Z^t\*u^\star$ and applied the communication error definition $\mathcal{E}^{c}_{t,k}:=\*x^t_k - \*Z^t\*y_k$ to get the first equality. We used the standard inequality $\pm2\langle a,b\rangle \leq c\|a\|^2 + c^{-1}\|b\|^2$ that holds for any two vectors $a,b$ and positive constant $c>0$ to get the first inequality.
Finally, we derive the last inequality using the relations $\left\|\*Z^t \right\|=1$ and $\left\|I-\*Z^t \right\|<2$ that hold due to Assumption~\ref{assum:consensus_mat}.

Due to the fact that $\mathcal{F}^0_k \subseteq \mathcal{F}^t_k$, combining the preceding three relations and taking the expectation conditional on $\mathcal{F}^0_k$ on both sides of~\eqref{eq:lem_bounded_xy_2_q2} yields,
 \begin{equation*}
     \begin{split}
        \mathbb{E}\left[\left\| \*y_{k+1} - \*u^\star \right\|^2 \Big| \mathcal{F}^0_k \right]&\leq \left(1-\nu^2\right) 
        \left\|\*y_k - \*u^\star \right\|^2 + 
        \alpha^2n\sigma_g^2 +  (1-\nu)\mathbb{E}\left[\left\|\mathcal{E}^c_{t,k}\right\|^2\big|\mathcal{F}^0_k\right]\\
        & \quad +  2(1-\nu)\mathbb{E}\left[\left \langle \mathcal{E}^{c}_{t,k}, \*Z^t\*y_k - \*u^\star  \right \rangle \big|\mathcal{F}^0_k\right]+ 4\nu^{-1}\left(1-\nu^2\right)\left\|\*u^\star\right\|^2\\
        &\leq \left(1-\nu^2\right) 
        \left\|\*y_k - \*u^\star \right\|^2 + 
        \alpha^2n\sigma_g^2 +  (1-\nu)nt\sigma_c^2+ 4\nu^{-1}\left(1-\nu^2\right)\left\|\*u^\star\right\|^2,
    \end{split}
 \end{equation*}
 where we applied Lemma~\ref{lem:comm_error_q2} to get the last inequality.
 
Taking the full expectation on both sides of the relation above and applying recursively over iterations $0, 1, \ldots, k$ yields,
\begin{equation*}
    \begin{split}
        \mathbb{E}\left[\left\| \*y_{k} - \*u^\star \right\|^2\right]
        &\leq \left(1-\nu^2\right)^k \mathbb{E}\left[
      \left\|\*y_0 - \*u^\star \right\|^2 \right]+
        \alpha^2\nu^{-2}n\sigma_g^2 +  (\nu^{-2}-\nu^{-1})nt\sigma_c^2+ 4\left(\nu^{-3}-\nu^{-1}\right)\left\|\*u^\star\right\|^2\\
        &\leq \mathbb{E}\left[
      \left\|\*y_0 - \*u^\star \right\|^2 \right]+
        \alpha^2\nu^{-2}n\sigma_g^2 +  \nu^{-2}nt\sigma_c^2 + 4\nu^{-3}\left\|\*u^\star\right\|^2,
        \end{split}
\end{equation*}
where we used $\sum_{h=0}^{k-1}\left(1-\nu^2\right)^h\leq \nu^{-2}$ to get the first inequality
and $\nu > 0$ to get the second inequality.
Moreover, the following statement holds
\begin{equation*}
\begin{split}
    \left\| \*y_{k}\right\|^2 &= \left\| \*y_{k} - \*u^\star + \*u^\star \right\|^2\\
    &\leq 2\left\| \*y_{k} - \*u^\star \right\|^2+2\left\| \*u^\star \right\|^2.
    \end{split}
\end{equation*}

After taking the total expectation on both sides of the above relation, we obtain
\begin{equation}
\label{eq:bound_yk_q2}
\begin{split}
    \mathbb{E}\left[\left\| \*y_{k}\right\|^2 \right]& \leq 2\mathbb{E}\left[\left\| \*y_{k} - \*u^\star \right\|^2\right]+2\left\| \*u^\star \right\|^2\\
    &\leq 2 \mathbb{E}\left[
      \left\|\*y_0 - \*u^\star \right\|^2 \right]+ 
      2 \alpha^2\nu^{-2}n\sigma_g^2 + 2 \nu^{-2}nt\sigma_c^2 + 2\left(1+4\nu^{-3}\right)\left\|\*u^\star\right\|^2.
    \end{split}
\end{equation}
Applying the definitions of $D$ and $\kappa$ to (\ref{eq:bound_yk_q2}) yields the first result of this lemma.

Finally, the following statement also holds 
\begin{equation*}
    \begin{split}
        \left\| \*x^t_k \right\|^2 &= \left\| \mathcal{E}^{c}_{t,k} + \*Z^t\*y_k\right\|^2\\
        &= \left\| \mathcal{E}^{c}_{t,k}\right\|^2 +\left\| \*Z^t\*y_k\right\|^2 + 2\left \langle\mathcal{E}^{c}_{t,k}, \*Z^t\*y_k \right \rangle\\
         &\leq \left\| \mathcal{E}^{c}_{t,k}\right\|^2 +\left\| \*y_k\right\|^2  + 2\left \langle\mathcal{E}^{c}_{t,k}, \*Z^t\*y_k \right \rangle,
    \end{split}
\end{equation*}
where we used the non-expansiveness of $\*Z$ for the last inequality.

Taking the expectation conditional  on $\mathcal{F}^0_k$ on both sides of the preceding relation and applying Lemma \ref{lem:comm_error_q2} yields,
\begin{equation}
\label{eq:bound_xk}
    \begin{split}
        \mathbb{E}\left[\left\| \*x^t_k \right\|^2 \Big | \mathcal{F}^0_k \right]  &\leq  nt\sigma_c^2 +\left\| \*y_k\right\|^2.
    \end{split}
\end{equation}
Taking the total expectation on both sides of the above relation, applying (\ref{eq:bound_yk_q2})  and the definitions of $D$ and $\kappa$ concludes this proof.

\end{proof}


\begin{lem}
\textbf{(Bounded distance to average (Q.2))}
\label{lem:bounded_variance_q2}
Let $x_{i,k}^t$ and $y_{i,k}$ be the local iterates produced by S-NEAR-DGD$^t$ at node $i$ and iteration count $k$ and let $\bar{x}_k^t:=\sum_{i=1}^n x^t_{i,k}$ and $\bar{y}_k:=\sum_{i=1}^n y_{i,k}$ denote the average iterates across all nodes. Then the distance between the local and average iterates is bounded in expectation for $i=1,...,n$ and $k=1,2,...$, namely
\begin{equation*}
\begin{split}
    \mathbb{E}\left[\left\|x^t_{i,k} - \bar{x}^t_k \right\|^2\right] &\leq \mathbb{E}\left[\left\|\*x^t_k - \*M \*x^t_k \right\|^2\right] \leq \beta^{2t}D + 
      \frac{\beta^{2t}(1+\kappa)^2n\sigma_g^2}{2L^2}+ \frac{\beta^{2t}(1+\kappa)^2nt\sigma_c^2}{2\alpha^2 L^2} + nt\sigma_c^2,
     \end{split}
\end{equation*}
and
\begin{equation*}
\begin{split}
    \mathbb{E} \left[\left\|y_{i,k} - \bar{y}_k \right\|^2 \right]&\leq \mathbb{E}\left[ \left\|\*y_k - \*M \*y_k \right\|^2 \right]\leq D + 
      \frac{(1+\kappa)^2n\sigma_g^2}{2L^2}+ \frac{(1+\kappa)^2nt\sigma_c^2}{2\alpha^2 L^2},
    \end{split}
\end{equation*}
where
$\*M = \left( \frac{1_n 1_n^T}{n} \otimes I_p \right) \in \mathbb{R}^{np}$ is the averaging matrix, constant $D$ is defined in Lemma~\ref{lem:bounded_iterates_q2} and $\kappa=L/\mu$ is the condition number of Problem~\eqref{eq:consensus_prob}. 
 
\end{lem}

\begin{proof}
Observing that $\sum_{i=1}^n \left\|x^t_{i,k}-\bar{x}^t_k \right\|^2 = \left\| \*x^t_k - \*M \*x^t_k \right\|^2$, we obtain
\begin{equation}
\label{eq:lem_bounded_diff_x_q2}
    \left\|x^t_{i,k}-\bar{x}^t_k \right\|^2 \leq \left\| \*x^t_k - \*M \*x^t_k \right\|^2,\forall i=1,...,n.
\end{equation}
We can bound the right-hand side of (\ref{eq:lem_bounded_diff_x_q2}) as
\begin{equation}
\label{eq:lem2_x1_q2}
    \begin{split}
        \left\|\*x^t_k - \*M \*x_k\right\|^2 &=  \left\|\mathcal{E}^{c}_{t,k} + \*Z^t\*y_k- \*M \*x_k - \*M\*y_k + \*M\*y_k \right\|^2\\
        &=\left\|\mathcal{E}^{c}_{t,k} + \left(\*Z^t-\*M\right)\*y_k- \*M\*x_k +\*M \*Z^t\*y_k \right\|^2\\
        &=\left\|\left(I-\*M\right)\mathcal{E}^{c}_{t,k} \right\|^2+\left\| \left(\*Z^t-\*M\right)\*y_k \right\|^2 + 2\left \langle\left(I-\*M\right)\mathcal{E}^{c}_{t,k} ,\left(\*Z^t-\*M\right)\*y_k  \right \rangle\\
        &\leq \left\|\mathcal{E}^{c}_{t,k} \right\|^2+\beta^{2t}\left\|\*y_k \right\|^2 + 2\left \langle\left(I-\*M\right)\mathcal{E}^{c}_{t,k} ,\left(\*Z^t-\*M\right)\*y_k  \right \rangle,
    \end{split}
\end{equation}
where we applied the definition of the communication error $\mathcal{E}^c_{t,k}$ of Lemma~\ref{lem:comm_error_q2}
and added and subtracted $\*M\*y_k$ to obtain the first equality, and used the fact that $\*M \*Z^t = \*M$ to get the second equality.
We derive the last inequality from Cauchy-Schwarz and the spectral properties of $\*Z^t = \*W^t \otimes I_p$ and $\*M = \left(\frac{1_n 1_n^T}{n}\right) \otimes I_p$.


Taking the expectation conditional to $\mathcal{F}^0_k$ on both sides of (\ref{eq:lem2_x1_q2}) and applying Lemma~\ref{lem:comm_error_q2} yields,
\begin{equation*}
    \begin{split}
      \mathbb{E}\left[ \left\|\*x^t_k - \*M \*x_k\right\|^2 \Big | \mathcal{F}^0_k \right] &\leq nt\sigma_c^2+\beta^{2t}\left\|\*y_k \right\|^2.
    \end{split}
\end{equation*}
Taking the total expectation on both sides and applying Lemma~\ref{lem:bounded_iterates_q2} yields the first result of this lemma.

Similarly, the following inequality holds for the $\*y_k$ iterates
\begin{equation}
\label{eq:lem_bounded_diff_y_q2}
    \left\|y_{i,k}-\bar{y}_k \right\|^2 \leq \left\| \*y_k - \*M \*y_k \right\|^2,\forall i=1,...,n.
\end{equation}
For the right-hand side of (\ref{eq:lem_bounded_diff_y_q2}), we have
\begin{equation*}
    \begin{split}
        \left\|\*y_k - \*M\*y_k \right\|^2 &= \left\| \left(I-\*M\right)\*y_k\right\|^2\\
        &\leq \left\| \*y_k \right\|^2,
\end{split}
\end{equation*}
where we have used the fact that $\left\|I-\*M\right\|=1$.

Taking the total expectation on both sides and applying Lemma~\ref{lem:bounded_iterates_q2} concludes this proof.

\end{proof}

\begin{thm}
\textbf{(Convergence of S-NEAR-DGD$^t$ (Q.2))}
\label{thm:bounded_dist_min_q2}
Let $\bar{x}^t_k:=\frac{1}{n}\sum_{i=1}^n x^t_{i,k}$ denote the average of the local iterates generated by~\eqref{eq:near_dgd_x_q2} from initial point $\*y_0$
at iteration $k$ and let the steplength $\alpha$ satisfy
\begin{equation*}
    \alpha < \min \left \{\frac{2}{\mu+L},\frac{2}{\mu_{\bar{f}}+L_{\bar{f}}} \right\},
\end{equation*}
where $\mu=\min_i L_i$, $L = \max_i L_i$, $\mu_{\bar{f}}=\frac{1}{n}\sum_{i=1}^n\mu_i$ and $L_{\bar{f}}=\frac{1}{n}\sum_{i=1}^n L_i$.

Then the distance of $\bar{x}^t_k$ to the optimal solution $x^\star$ of Problem~(\ref{eq:consensus_prob}) is bounded in expectation for $k=1,2,...$,
\begin{equation}
\label{eq:thm_xk1_q2}
    \begin{split}
    \mathbb{E}\left[\left\|\bar{x}^t_{k+1} - x^\star\right\|^2  \right]
    &\leq \rho \mathbb{E}\left[ \left\| \bar{x}^t_{k} - x^\star\right\|^2\right] + \frac{\alpha\beta^{2t}\rho L^2D}{n\gamma_{\bar{f}}} + \frac{\alpha^2 \sigma_g^2  }{n} \\
    &\quad +  \frac{\alpha\beta^{2t}\left(1+\kappa\right)^2\rho\sigma_g^2}{2\gamma_{\bar{f}}} + \frac{t\sigma_c^2}{n} + \frac{\alpha\rho L^2t\sigma_c^2}{\gamma_{\bar{f}}} +  \frac{\beta^{2t}\left(1+\kappa\right)^{2}\rho t\sigma_c^2}{2\alpha\gamma_{\bar{f}}},
    \end{split}
\end{equation}
and
\begin{equation}
\label{eq:thm_xk2_q2}
    \begin{split}
    \mathbb{E}\left[\left\|\bar{x}^t_{k} - x^\star\right\|^2  \right]
    &\leq  \rho^k \mathbb{E}\left[ \left\| \bar{x}_0 - x^\star\right\|^2\right] + \frac{\beta^{2t}\rho L^2D}{n\gamma_{\bar{f}}^2} + \frac{\alpha \sigma_g^2  }{n\gamma_{\bar{f}}} \\
    &\quad +  \frac{\beta^{2t} \left(1+\kappa \right)^2\rho\sigma_g^2}{2\gamma_{\bar{f}}^2} + \frac{t\sigma_c^2}{n\alpha\gamma_{\bar{f}}} + \frac{\rho L^2t\sigma_c^2}{\gamma_{\bar{f}}^2} +  \frac{\beta^{2t}\left(1+\kappa\right)^2\rho t\sigma_c^2}{2\alpha^2\gamma_{\bar{f}}^2},
    \end{split}
\end{equation}
where $\bar{x}_0=\frac{1}{n}\sum_{i=1}^n y_{i,0}$, $\rho=1-\alpha\gamma_{\bar{f}}$, $\gamma_{\bar{f}}=\frac{\mu_{\bar{f}}L_{\bar{f}}}{\mu_{\bar{f}}+L_{\bar{f}}}$, $\kappa=L/\mu$ is the condition number of Problem~\eqref{eq:consensus_prob} and the constant $D$ is defined in Lemma~\ref{lem:bounded_iterates_q2}. 

\end{thm}

\begin{proof}
Applying ($\ref{eq:near_dgd_x_q2}$) to the $(k+1)^{th}$ iteration and calculating the average iterate $\bar{x}^j_{k+1}=\frac{1}{n}\sum_{i=1}x_{i,k+1}^j$ across all $n$ nodes, we obtain,
\begin{equation*}
\begin{split}
    \bar{x}^j_{k+1} &=  \bar{q}^{j}_{k+1}\text{, }j=1,...,t,
    \end{split}
\end{equation*}
where $\bar{q}^{j}_{k+1} = \frac{1}{n}\sum_{i=1}^n q^j_{i,k+1} = \frac{1}{n}\sum_{i=1}^n \left( \epsilon^{j}_{i,k+1} + x^{j-1}_{i,k+1}\right)$ and $x^0_{i,k+1}=y_{i,k+1}$.

Hence, we can show that,
\begin{equation*}
    \begin{split}
         \bar{x}^t_{k+1} - \bar{y}_{k+1} &=  \bar{q}^t_{k+1} -\sum_{j=1}^{t-1}\bar{q}^j_{k+1} + \sum_{j=1}^{t-1}\bar{q}^j_{k+1}-\bar{x}^{0}_{k+1} \\
          &=  \sum_{j=1}^t\left(\bar{q}^j_{k+1} -\bar{x}^{j-1}_{k+1}\right)\\
         &= \frac{1}{n}\sum_{j=1}^{t}\sum_{i=1}^n   \epsilon^j_{i,k+1},
    \end{split}
\end{equation*}
where we added and subtracted the quantity $\sum_{j=1}^{t}\bar{q}^j_{k+1}$ to obtain the first equality. 

Then the following holds for the distance of $\bar{x}^t_{k+1}$ to $x^\star$,
\begin{equation}
\label{eq:lem_bounded_min200_q2}
    \begin{split}
    \left\|\bar{x}^t_{k+1} - x^\star\right\|^2 &=\left\|\bar{x}^t_{k+1} - \bar{y}_{k+1} + \bar{y}_{k+1}- x^\star\right\|^2\\
    &\leq \frac{1}{n^2}\left\|\sum_{j=1}^{t}\sum_{i=1}^n   \epsilon^j_{i,k+1}\right\|^2 + \left\| \bar{y}_{k+1} - x^\star\right\|^2 + \frac{2}{n}\left\langle \sum_{j=1}^{t}\sum_{i=1}^n   \epsilon^j_{i,k+1}, \bar{y}_{k+1} - x^\star\right \rangle,
    \end{split}
\end{equation}
where we added and subtracted $\bar{y}_{k+1}$ to get the first equality.

By linearity of expectation and Assumption~\ref{assum:tc_bound} we obtain,
\begin{equation*}
    \mathbb{E}\left[\sum_{j=1}^{t}\sum_{i=1}^n   \epsilon^j_{i,k+1}\bigg|\mathcal{F}^0_{k+1}\right] = \mathbf{0}.
\end{equation*}

Moreover, observing that $\mathbb{E}\left[\left\langle \epsilon^{j_1}_{i_1,k+1}, \epsilon^{j_2}_{i_2,k+1} \right \rangle\big|\mathcal{F}^0_{k+1}\right]=\mathbf{0}$ for $j_1\neq j_2$ and $i_1 \neq i_2$ due to Assumption~\ref{assum:tc_bound}, we acquire,
\begin{equation*}
\begin{split}
    \frac{1}{n^2}\mathbb{E}\left[\left\|\sum_{j=1}^{t}\sum_{i=1}^n   \epsilon^j_{i,k+1}\right\|^2\Bigg|\mathcal{F}^0_{k+1}\right] &=  \frac{1}{n^2}\sum_{j=1}^{t}\sum_{i=1}^n  \mathbb{E}\left[\left\| \epsilon^j_{i,k+1}\right\|^2\bigg|\mathcal{F}^0_{k+1}\right]\leq \frac{t\sigma_c^2}{n},
    \end{split}
\end{equation*}
where we applied Assumption~\ref{assum:tc_bound} to get the last inequality.

Taking the expectation conditional on $\mathcal{F}^0_{k+1}$ on both sides of~(\ref{eq:lem_bounded_min200_q2}) and substituting the two preceding relations yield,
\begin{equation}
\label{eq:lem_bounded_min201_q2}
    \begin{split}
    \mathbb{E}\left[\left\|\bar{x}^t_{k+1} - x^\star\right\|^2 \Big| \mathcal{F}^0_{k+1}\right] 
    &\leq  \left\| \bar{y}_{k+1} - x^\star\right\|^2 + \frac{t\sigma_c^2}{n}.
    \end{split}
\end{equation}
Due to the fact that $\mathcal{F}^t_{k} \subseteq \mathcal{F}^0_{k+1}$, taking the expectation conditional to $\mathcal{F}^t_{k}$  on both sides of (\ref{eq:lem_bounded_min201_q2}) and applying Lemma~\ref{lem:descent_avg_iter} and the tower property of expectation yields
 \begin{equation*}
    \begin{split}
    \mathbb{E}\left[\left\|\bar{x}^t_{k+1} - x^\star\right\|^2 \Big| \mathcal{F}^t_{k}\right] 
    &\leq   \rho  \left\| \bar{x}^t_{k} - x^\star\right\|^2 + \frac{\alpha^2 \sigma_g^2}{n} +  \frac{\alpha \rho L^2 \left\|\*x^t_k - \*M \*x^t_k\right\|^2}{n\gamma_{\bar{f}}} + \frac{t\sigma_c^2}{n}.
    \end{split}
\end{equation*}
After taking the full expectation on both sides of the relation above and applying Lemma~\ref{lem:bounded_variance_q2} we acquire,
\begin{equation}
\label{eq:lem_bounded_min6_q2}
    \begin{split}
    \mathbb{E}\left[\left\|\bar{x}^t_{k+1} - x^\star\right\|^2  \right]
    &\leq \rho \mathbb{E}\left[ \left\| \bar{x}^t_{k} - x^\star\right\|^2\right] + \frac{\alpha\beta^{2t}\rho L^2D}{n\gamma_{\bar{f}}} + \frac{\alpha^2\sigma_g^2  }{n} +  \frac{\alpha\beta^{2t}(1+\kappa)^{2}\rho \sigma_g^2}{2\gamma_{\bar{f}}} + \frac{t\sigma_c^2}{n} + \frac{\alpha\rho L^2t\sigma_c^2}{\gamma_{\bar{f}}} +  \frac{\beta^{2t}(1+\kappa)^{2}\rho t\sigma_c^2}{2\alpha\gamma_{\bar{f}}}.
    \end{split}
\end{equation}
We notice that $\rho<1$ and therefore after applying \eqref{eq:lem_bounded_min6_q2} recursively and using the relation $\sum_{h=0}^{k-1} \rho^h \leq (1-\rho)^{-1}$ we obtain,
\begin{equation*}
    \begin{split}
    \mathbb{E}\left[\left\|\bar{x}^t_{k} - x^\star\right\|^2  \right]
    \leq & \rho^k \mathbb{E}\left[ \left\| \bar{x}_0- x^\star\right\|^2\right] + \frac{\alpha\beta^{2t}\rho L^2D}{n\gamma_{\bar{f}}(1-\rho)} + \frac{\alpha^2 \sigma_g^2  }{n(1-\rho)} \\
    &\quad +  \frac{\alpha\beta^{2t}(1+\kappa)^{2}\rho \sigma_g^2}{2\gamma_{\bar{f}}(1-\rho)} + \frac{t\sigma_c^2}{n(1-\rho)} + \frac{\alpha\rho L^2t\sigma_c^2}{\gamma_{\bar{f}}(1-\rho)} +  \frac{\beta^{2t}(1+\kappa)^{2}\rho t\sigma_c^2}{2\alpha\gamma_{\bar{f}}(1-\rho)}.
    \end{split}
\end{equation*}
Applying the relation $1-\rho=\alpha \gamma_{\bar{f}}$ completes the proof.

\end{proof}




\begin{cor}
\textbf{(Convergence of local iterates (Q.2))}
\label{cor:local_dist_q2}
Let $y_{i,k}$ and $x^t_{i,k}$ be the local iterates generated by~\eqref{eq:near_dgd_y_q2} and~\eqref{eq:near_dgd_x_q2}, respectively, from initial point $\*x_0=\*y_0=[y_{1,0};...;y_{n,0}] \in \mathbb{R}^{np}$ and let the steplength $\alpha$ satisfy 
\begin{equation*}
    \alpha < \min \left \{\frac{2}{\mu+L},\frac{2}{\mu_{\bar{f}}+\L_{\bar{f}}} \right\},
\end{equation*}
where $\mu=\min_i\mu_i$, $L=\max_i L_i$, $\mu_{\bar{f}}=\frac{1}{n}\sum_{i=1}^n\mu_i$ and $L_{\bar{f}}=\frac{1}{n}\sum_{i=1}^n L_i$.

Then for $i=1,...,n$ and $k\geq1$ the distance of the local iterates to the solution of Problem~(\ref{eq:consensus_prob}) is bounded, i.e.,
\begin{equation*}
    \begin{split}
        \mathbb{E}\left[\left \|x^t_{i,k} - x^\star \right\|^2 \right]
        \leq &2\rho^k\mathbb{E} \left[\left\| \bar{x}_0 - x^\star\right\|^2\right] + 2\beta^{2t}\left(1+\frac{C }{n}\right)D +\frac{2\alpha\sigma_g^2}{n\gamma_{\bar{f}}} \\
      &\quad + \frac{\beta^{2t}\left(1+\kappa\right)^2\left(n+C\right)\sigma_g^2}{L^2} + 2\left(n+C\right)t\sigma_c^2+ \frac{\beta^{2t}\left(1+\kappa\right)^2\left(n+ C\right)t\sigma_c^2}{\alpha^2L^2} +  \frac{2t\sigma_c^2}{n\alpha\gamma_{\bar{f}}},
      \end{split}
\end{equation*}
and
\begin{equation*}
    \begin{split}
        \mathbb{E}\left[\left \|y_{i,k+1} - x^\star \right\|^2 \right]
        &\leq 2\rho^{k+1} \mathbb{E}\left[ \left\| \bar{x}_0 - x^\star\right\|^2\right] + 2D+ \frac{2\beta^{2t}CD}{n} + \frac{2\alpha \sigma_g^2}{n \gamma_{\bar{f}}^2} \\
      &\quad + \frac{(1+\kappa)^2\left(n+\beta^{2t}C\right)\sigma_g^2}{L^2} + 2Ct\sigma_c^2+ \frac{(1+\kappa)^2 \left(n + \beta^{2t}C\right) t\sigma_c^2}{\alpha^2 L^2} + \frac{2\rho t\sigma_c^2}{n \alpha \gamma_{\bar{f}}},
        \end{split}
\end{equation*}
where $C= \frac{\rho L^2}{\gamma_{\bar{f}}^2}$ $\rho=1-\alpha\gamma_{\bar{f}}$, $\gamma_{\bar{f}}=\frac{\mu_{\bar{f}}L_{\bar{f}}}{\mu_{\bar{f}}+L_{\bar{f}}}$, $\bar{x}_0=\sum_{i=1}^n y_{i,0}$ and the constant $D$ is defined in Lemma~\ref{lem:bounded_iterates_q2}.
\end{cor}

\begin{proof}
For $i=1,...,n$ and $k\geq1$, for the $x^t_{i,k}$ iterates we have,
\begin{equation} 
\label{eq:lem_dist_min_local1_q2}
    \begin{split}
        \left \|x^t_{i,k} - x^\star \right\|^2 &=  \left \|x^t_{i,k} -\bar{x}^t_k + \bar{x}^t_k - x^\star \right\|^2 \\
        & \leq  2\left \|x^t_{i,k} -\bar{x}^t_k\right\|^2 +2\left\| \bar{x}^t_k - x^\star \right\|^2,
    \end{split}
\end{equation}
where we added and subtracted $\bar{x}^t_k$ to get the first equality.

Taking the total expectation on both sides of (\ref{eq:lem_dist_min_local1_q2}) and applying Lemma~\ref{lem:bounded_variance_q2} and Theorem~\ref{thm:bounded_dist_min_q2} yields the first result of this corollary.

Similarly, for the $y_{i,k}$ local iterates we have,
\begin{equation}
\label{eq:lem_dist_min_local3_q2}
    \begin{split}
        \left \|y_{i,k+1} - x^\star \right\|^2 &=  \left \|y_{i,k+1}-\bar{y}_{k+1} + \bar{y}_{k+1} - x^\star \right\|^2 \\
        &\leq 2 \left\|y_{i,k+1}-\bar{y}_{k+1}\right\|^2 +2\left\| \bar{y}_{k+1} - x^\star \right\|^2,
        \end{split}
\end{equation}
where we derive the first equality by adding and subtracting $\bar{y}_{k+1}$. 

Moreover, taking the total expectation on both sides of the result of Lemma~\ref{lem:descent_avg_iter} yields,
\begin{equation*}
\label{eq:lem_local_helper}
    \begin{split}
      \mathbb{E}\left[ \left\| \bar{y}_{k+1} - x^\star\right\|^2 \right]
      &\leq  \rho  \mathbb{E}\left[ \left\| \bar{x}^t_{k} - x^\star\right\|^2\right] + \frac{\alpha^2 \sigma_g^2}{n} +  \frac{\alpha \rho L^2 }{n\gamma_{\bar{f}}} \mathbb{E} \left[\left\|\*x^t_k - \*M \*x^t_k\right\|^2\right]\\
      &\leq  \rho \bigg(\rho^k \mathbb{E}\left[ \left\| \bar{x}_0 - x^\star\right\|^2\right] + \frac{\beta^{2t}CD}{n} + \frac{\alpha \sigma_g^2  }{n\gamma_{\bar{f}}}  +  \frac{\beta^{2t} \left(1+\kappa \right)^2\rho\sigma_g^2}{2\gamma_{\bar{f}}^2} + \frac{t\sigma_c^2}{n\alpha\gamma_{\bar{f}}} + Ct\sigma_c^2 +  \frac{\beta^{2t}\left(1+\kappa\right)^2\rho t\sigma_c^2}{2\alpha^2\gamma_{\bar{f}}^2} \bigg)\\
    &\quad + \frac{\alpha^2 \sigma_g^2}{n} + \frac{\alpha C \gamma_{\bar{f}} }{n}\bigg(\beta^{2t}D + 
      \frac{\beta^{2t}(1+\kappa)^2n\sigma_g^2}{2L^2}+ \frac{\beta^{2t}(1+\kappa)^2nt\sigma_c^2}{2\alpha^2 L^2} + nt\sigma_c^2 \bigg)\\
      &=\rho^{k+1} \mathbb{E}\left[ \left\| \bar{x}_0 - x^\star\right\|^2\right] + \frac{\beta^{2t}CD}{n}\left(\rho + \alpha \gamma_{\bar{f}}\right) + \frac{\alpha \sigma_g^2}{n}\left(\alpha + \frac{\rho}{\gamma_{\bar{f}}}\right) +  \frac{\beta^{2t} \left(1+\kappa \right)^2\sigma_g^2}{2} \left(\frac{\rho^2}{\gamma_{\bar{f}}^2} + \frac{\alpha C \gamma_{\bar{f}}}{L^2} \right) \\
      &\quad + t\sigma_c^2 \left(\frac{\rho}{n \alpha \gamma_{\bar{f}}} + \rho C + \alpha C \gamma_{\bar{f}} \right) + \frac{\beta^{2t}(1+\kappa)^2 t \sigma_c^2}{2\alpha^2} \left( \frac{\rho^2}{\gamma_{\bar{f}}^2} + \frac{\alpha C \gamma_{\bar{f}}}{L^2}\right)\\
      &=\rho^{k+1} \mathbb{E}\left[ \left\| \bar{x}_0 - x^\star\right\|^2\right] + \frac{\beta^{2t}CD}{n} + \frac{\alpha \sigma_g^2}{n \gamma_{\bar{f}}^2} \\
      &\quad +  \frac{\beta^{2t} \left(1+\kappa \right)^2 \rho \sigma_g^2}{2\gamma_{\bar{f}}^2}  + t\sigma_c^2 \left(\frac{\rho}{n \alpha \gamma_{\bar{f}}} + C \right) + \frac{\beta^{2t}(1+\kappa)^2 \rho t \sigma_c^2}{2\alpha^2 \gamma_{\bar{f}}^2},
    \end{split}
\end{equation*}
where we applied Theorem~\ref{thm:bounded_dist_min_q2}, Lemma~\ref{lem:bounded_variance_q2} and the definition of $C$ to get the second inequality. We derive the remaining equalities by algebraic manipulation.

Taking the total expectation on both sides of (\ref{eq:lem_dist_min_local3_q2}), applying the inequality above and Lemma~\ref{lem:bounded_variance_q2} completes the proof.
\end{proof}


\begin{thm}\textbf{(Distance to minimum, NEAR-DGD$^+$ (Q.2))}
\label{thm:near_dgd_plus_q2}
Consider the S-NEAR-DGD$^+$ method under consensus variation Q.2, i.e. $t(k)=k$ in~\eqref{eq:near_dgd_y_q2} and~\eqref{eq:near_dgd_x_q2}. Let $\bar{x}^k_k=\frac{1}{n}\sum_{i=1}^n x^k_{i,k}$ be the average $x_{i,k}^k$ iterates produced by S-NEAR-DGD$^+$ and let the steplength $\alpha$ satisfy,
\begin{equation*}
    \alpha < \min \left \{\frac{2}{\mu+L},\frac{2}{\mu_{\bar{f}}+L_{\bar{f}}} \right\}.
\end{equation*}
Then the following inequality holds for the distance of $\bar{x}_k$ to $x^\star$ for $k=1,2,...$
\begin{equation*}
    \begin{split}
        \mathbb{E}\left[\left\|\bar{x}^k_k - x^\star \right\|^2\right] &\leq \rho^k \mathbb{E} \left[ \left\|\bar{x}_0 - x^\star \right\|^2 \right]+ \frac{\eta\theta^k\alpha\rho L^2D}{n\gamma_{\bar{f}}} + \frac{\alpha\sigma_g^2}{n\gamma_{\bar{f}}}\\
        &\quad+ \frac{\eta\theta^k\alpha\left(1+\kappa\right)^{2} \rho \sigma_g^2}{2\gamma_{\bar{f}}} + \frac{(k-1) \sigma_c^2}{n\alpha\gamma_{\bar{f}}} + \frac{\rho L^2 (k-1)\sigma_c^2}{\gamma_{\bar{f}}^2}  + \frac{\eta\theta^k\left(1+\kappa\right)^{2} \rho (k-1)\sigma_c^2}{2\alpha\gamma_{\bar{f}}},
    \end{split}
\end{equation*} 
where $\eta= \left|\beta^2-\rho\right|^{-1}$ and $\theta = \max\left\{\rho,\beta^2\right\}$.
\end{thm}

\begin{proof}
Replacing $t$ with $k$ in (\ref{eq:thm_xk1_q2}) in Theorem~\ref{thm:bounded_dist_min_q2} yields,
\begin{equation*}
    \begin{split}
     \mathbb{E}\left[\left\|\bar{x}^{k+1}_{k+1} - x^\star\right\|^2  \right]
    &\leq  \rho \mathbb{E}\left[\left\| \bar{x}^k_{k} - x^\star\right\|^2\right] + \frac{\alpha\beta^{2k}\rho L^2D}{n\gamma_{\bar{f}}} +\frac{\alpha^2\sigma_g^2}{n}\\
    &\quad + \frac{\alpha\beta^{2k}\left(1+\kappa\right)^{2}\rho\sigma_g^2}{2\gamma_{\bar{f}}}+ \frac{k\sigma_c^2}{n} + \frac{\alpha\rho L^2k\sigma_c^2}{\gamma_{\bar{f}}} + \frac{\beta^{2k}\left(1+\kappa\right)^{2}\rho k\sigma_c^2}{2\alpha\gamma_{\bar{f}}}.
    \end{split}
\end{equation*}
Applying recursively for iterations $1, 2,\ldots, k$, we obtain,
\begin{equation}
\label{eq:thm_near_dgd_p_1_q2}
    \begin{split}
        \mathbb{E}\left[\left\|\bar{x}^k_k - x^\star \right\|^2\right] &\leq \rho^k \mathbb{E} \left[ \left\|\bar{x}_0 - x^\star \right\|^2 \right]+ S_1\left(\frac{\alpha\rho L^2D}{n\gamma_{\bar{f}}} + \frac{\alpha\left(1+\kappa\right)^2\rho \sigma_g^2}{2\gamma_{\bar{f}}} \right)  \\
        &\quad + S_2\left(\frac{\alpha^2\sigma_g^2}{n}\right)  + S_3\left( \frac{\sigma_c^2}{n} + \frac{\alpha\rho L^2\sigma_c^2}{\gamma_{\bar{f}}}  \right) + S_4\left(\frac{\left(1+\kappa\right)^2\rho \sigma_c^2}{2\alpha\gamma_{\bar{f}}}\right).
    \end{split}
\end{equation}
where 
\begin{gather*}
    S_1=\sum_{j=0}^{k-1} \rho^j \beta^{2(k-1-j)}, \quad S_2 = \sum_{j=0}^{k-1}\rho^j\\
    S_3 = \sum_{j=0}^{k-1}(k-1-j) \rho^{j}, \quad S_4 = \sum_{j=0}^{k-1} (k-1-j)\rho^j\beta^{2(k-1-j)}.
\end{gather*}
Let $\psi = \frac{\rho}{\beta^2}$. We can bound the first two sums with $ S_1 = \beta^{2(k-1)} \sum_{j=0}^{k-1} \psi^j
    = \beta^{2(k-1)}\frac{1-\psi^k}{1-\psi}
     = \frac{\beta^{2k}-\rho^k}{\beta^2 - \rho}
    \leq \eta \theta^k$ and $S_2 = \frac{1-\rho^k}{1-\rho} \leq \frac{1}{1-\rho}=\frac{1}{\alpha \gamma_{\bar{f}}}$.
For the third sum we obtain $S_3=(k-1)S_2 -  \sum_{j=0}^{k-1}j \rho^{j} 
          \leq (k-1)S_2
          = \frac{k-1}{\alpha \gamma_{\bar{f}}}$.
Finally, we can derive an upper bound for the final sum using $S_4= (k-1)S_1 - \sum_{j=0}^{k-1} j\rho^j\beta^{2(k-1-j)}\leq (k-1)S_1\leq  (k-1) \eta \theta^k$.

Substituting the sum bounds in (\ref{eq:thm_near_dgd_p_1_q2}) yields the final result of this theorem.
\end{proof}

\bibliographystyle{ieeetr}
\bibliography{bibtex/prospectus} 




\end{document}